\documentclass[a4paper,12pt]{article}

\usepackage[english]{babel}
\usepackage[latin1]{inputenc}
\usepackage[T1]{fontenc}
\usepackage{amsmath,amssymb}
\usepackage{amsfonts}
\usepackage{amsthm}
\usepackage{delarray}

\theoremstyle{plain}
\newtheorem{theorem}{Theorem}[section]
\newtheorem{lemma}[theorem]{Lemma}
\newtheorem{proposition}[theorem]{Proposition}
\newtheorem{cor}{Corollary}
\newtheorem{appl}{Application}[section]
\newtheorem*{theorem*}{Theorem}
\newtheorem*{appl*}{Application}

\theoremstyle{definition}
\newtheorem{definition}{Definition}[section]

\theoremstyle{remark}
\newtheorem*{remark}{Remark}

\numberwithin{equation}{section}
\newcommand\dr{\mathbb}
\renewcommand\d{\textrm}

\newcommand\st{\textrm{ such that }}
\newcommand\bG{{\bf G}}
\newcommand\G{\Gamma}
\newcommand\g{\gamma}

\newcommand\GG{{G / \G}}
\newcommand\GH{{H  \backslash G}}
\renewcommand\phi{\varphi}

\title{Polynomial dynamic and lattice orbits in $S$-arithmetic homogeneous spaces}

\author{Antonin Guilloux}

\begin{document}

\maketitle

\section{Introduction}

Consider an homogeneous space under a locally compact group $G$ and a lattice $\G$ in $G$. Then the lattice naturally acts on the homogeneous space. Looking at a dense orbit, one may wonder how to describe its repartition. One then adopt a dynamical point of view and compare the asymptotic distribution of points in the orbits with the natural measure on the space. In the setting of Lie groups and their homogeneous spaces, several results we will present afterwards showed an equidistribution of points in the orbits.

We address here this problem in the setting of $p$-adic and $S$-arithmetic groups.

\subsection{Historical background}

Ten years ago, F. Ledrappier \cite{ledrappier} explained how Ratner's theory shall be used to understand the asymptotic properties of the action of $\d{SL}(2,\dr Z)$ on the euclidean plane $\dr R^2$. He proved the following:
\begin{theorem}[Ledrappier \cite{ledrappier}]\label{the:led}
Let $\G$ be a lattice of $\d{SL}(2,\dr R)$ of covolume $c(\G)$, $\|.\|$ the euclidean norm on the algebra of $2\times 2$-matrices $\mathcal M(2,\dr R)$, and $v\in \dr R^2$ with non-discrete orbit under $\G$.

Then we have the following limit, for all $\phi\in \mathcal C_c(\dr R^2\setminus\{0\})$:
$$\frac{1}{T}\sum_{\g\in \G\;, \|\g\|\leq T} \phi(\g v)\xrightarrow{T\to \infty} \frac{1}{|v|c(\G)}\int_{\dr R^2\setminus\{0\}} \phi(w) \frac{dw}{|w|} \; .$$
\end{theorem}

\begin{remark}
Nogueira \cite{nogueira} proved also the previous theorem for $\G=SL(2,\dr Z)$ using different techniques.
\end{remark}

After that A. Gorodnik develloped the strategy for the space of frames \cite{gorodnik-frames} and eventually A. Gorodnik and B. Weiss gave an abstract theorem for this problem in Lie groups and then applied it to different situations \cite{goroweiss}.

Recently F. Ledrappier and M. Pollicott \cite{ledrappier-pollicott}, and independently the author in its PhD thesis \cite{mathese}, proved a $p$-adic analog of the first theorem for lattices of $\d{SL}(2,\dr Q_p)$ acting on the $p$-adic plane. 

In this paper we adapt this strategy to handle the case of homogeneous space under $S$-arithmetic groups. Our work  can be viewed as the analog of \cite{goroweiss} in this setting.

\subsection{The $S$-arithmetic setting}

We will work in the following arithmetic setting:
let $K$ be a number field, $\mathcal O$ its integer ring and $\mathcal V$ the set of its places. We fix a finite set $S$ in $\mathcal V$ containing the archimedean ones. For all $\nu \in \mathcal V$, we note $K_\nu$ the completion of $K$ associated to $\nu$ and $K_S$ the module product of all $K_\nu$ for $\nu\in S$. This ring has a set of integer, noted $\mathcal O_S$.

Consider $\bG$ a semisimple simply connected $K$-group. We note $G:={\bf G}(K_S)$ its $S$-points, and we fix $\G$ an arithmetic lattice - i.e. commensurable to $\bG(\mathcal O_S)$. Recall that, according to Margulis superrigidity theorem, as soon has the total rank of $G$ is greater than $2$, any lattice in $G$ is an arithmetical one. Then let $H$ be a subgroup of $G$ which is a product $\displaystyle \prod_{\nu\in S}H_\nu$ of closed subgroups of $\bG(K_\nu)$. For example, one can think to the stabilizer of a point for an action of $G$ defined over $K$, i.e. $H=g\bar H g^{-1}$ where $\bar H$ is the $K_S$-points of a $K$-group and $g$ an element in $G$. We will always assume that the subgroup $H$ is unimodular. Some references for these objects are to be found in \cite{platonov-rapinchuk} and \cite{margulis}.

We are interested in the asymptotic distribution of orbits of $\G$ in $\GH$ so we will always assume this orbit to be dense, or equivalently that $H\G$ is dense in $G$. This last asumption is quite different of some recent works in the same area (\cite{gorodnik-oh}, \cite{venkatesh-ellenberg}...) where $H$ is supposed to have a closed projection in $\GG$ and the dynamic appears by looking at larger and larger orbits. In particular, there won't be any adelic arguments in this work.

\subsubsection{Measures and projections}

\begin{definition}
We say that a triple $(G,H,\G)$ is \emph{under study} if we are in the precedent case, that is if there is a number field $K$, a finite set $S$ of places containing the archimedean ones,
 and a $K$-group $\bG$, $K$-reductive and with simply connected semisimple part, such that :
\begin{itemize}
\item $G$ is the $K_S$ points of $\bG$,
\item $\G$ is an arithmetic lattice in $G$,
\item $H$ is the product of unimodular $K_\nu$-subgroups of $\bG(K_\nu)$ for $\nu \in S$,
\item $H\G$ is dense in $G$ and $H$ is not compact,
\item $H$ is a semidirect product $H^{ss}\rtimes H^u$ of a semisimple part and an unipotent radical.
\end{itemize}
\end{definition}

We now fix some notations for projections and measures : the Haar measure on $G$ is noted $m_G$ ; on $H$, $m_H$ ; and $m$ the probability on $\GG$ locally proportional to $m_G$. On $\GH$, as $H$ is unimodular, we have a unique - up to scaling - $G$-invariant measure. We normalize the measure $m_\GH$ on $\GH$ such that $m_G$ is locally the product of $m_H$ and $m_\GH$. The notations for the projections are as shown:
$$
\begin{matrix}
& G& \\
   \tau \swarrow   &  &  \searrow \pi      \\
\GH &   & \GG\\
\end{matrix}
$$

\subsubsection{Balls and volume}

In order to adopt a dynamical point of view, we need to instillate some evolution in the so far static situation. So we consider families $(G_t)_{t\in R}$ of open and bounded subsets in $G$ (often called balls), and consider the sets $\G_t=\G\cap G_t$. Letting $t$ go to $\infty$, we may now consider the asymptotic distribution of the sets $H\backslash H\G_t$ in $\GH$. Of course we will usually consider family $(G_t)$ that are increasing and exhausting (the union of $G_t$ covers $G$).

We introduce a notation for the intersection of such a family $(G_t)$ and its translates with subsets of $G$:
\begin{definition}\label{def:skewball}
Fix $(G_t)_{t\in\dr R}$ a family of open subset $G$, $L$ a subset of $G$ and $g$ an element of $G$. Then for all real $t$, we note $L_t:=L\cap G_t$ the intersection of $G_t$ ad $L$ and $L_t(g)$ the intersection $L\cap G_t g^{-1}$.

As the restriction of the so-called balls of $G$, we call the sets $L_t$ \emph{balls} in $L$, and \emph{skew-balls} the sets $L_t(g)$.
\end{definition}

When $L$ is a subgroup, we can compare the growth of volume of its normal subgroup with respect to the sets $(G_t)$. It may happens that a strict subgroup grows as fast as the whole group. Such a subgroup is exhibited in \cite[Section 12.3]{goroweiss}. We will call such a subgroup dominant:
\begin{definition}
Let $L$ be a unimodular subgroup of $G$ and $m_L$ be its Haar measure. Fix $G_t$ a family of open bounded subsets of $G$, increasing and exhausting.

A normal subgroup $L'$ is said to be dominant in $L$ if for some compact $C$ in $L$, the volume of $C.L'_t$ grows as fast as the volume of $L_t$, i.e. $\frac{m_L(C.L'_t)}{m_L(L_t)}$ does not converge to $0$ with $t$.
\end{definition}

Eventually we need an explicit way to define balls in $\G$. Going back to Ledrappier's theorem, we see that the balls are constructed considering a norm on the algebra of matrices. Moreover, Gorodnik and Weiss \cite{goroweiss} defined their balls in the same spirit, first representing the group $G$ and then using a norm on the matrix algebra in which $G$ is embedded. Our strategy is the same, but for technical reasons we assume firstly that the unipotent radical and the semisimple part are somehow orthogonal with respect to the norm and secondly that the norms are "algebraic".

\begin{definition}
A \emph{size function} $D$ from $G$ to $\dr R_+$ is any function constructed in the following way : consider a $K$-representation $\rho$ of $\bG$ in a space ${\bf V}$ and for all $\nu\in S$ a norm $|\, .\, |_\nu$ on the space $\d{End}({\bf V}(K_\nu))$ verifying :
\begin{enumerate}
\item for all $h_\nu=(h_\nu^{ss},h_\nu^u)$ in $H$, its norm $|h_\nu|_\nu$ is an increasing function of both $|h^{ss}_\nu|_\nu$ and $|h^u_\nu|_\nu$.
\item If $\nu$ is archimedean, the norm $|.|_\nu$ may be written in a suitable basis as the $L_p$-norm for $p$ in $\dr N^*\cup \{\infty\}$. If $\nu$ is ultrametric, we assume that it is the max-norm in some basis.
\end{enumerate}
Now define $D$ for all $g=(g_\nu)_{\nu \in S}$ by the formula $D(g)=\max\{|g_\nu|_\nu \d{ for } \nu \in S\}$.
\end{definition}

In this setting given a size function, we have a family of open bounded subsets $G_t := \{g\in G\st F(g)< t\}$ in $G$. 

\begin{remark}
These two assumptions, especially the first one, are annoying. The second one does not seem to be an important one and in numerous applications our work may be applied without it. For the first one, I do not know wether it is necessary or not. The positive point is that for applications we may verify it (see section \ref{sec:exa}): e.g. there is no condition when $H$ is either unipotent or semisimple. Morever every example given in the historical section fit into the framework of our article.
\end{remark}

\subsection{Statement of the main result}

We prove in this article the following result:
\begin{theorem}\label{the:normball}
Let $(G,H,\G)$ be a triple under study, $D$ be a size function on $G$ and $(G_t)_{t>0}$ be the associated family of balls. Assume that every dominant subgroup $H'$ verifies $H'\G$ is dense in $G$.

Then there is a finite partition $I_1,\ldots, I_l$ of $\dr R_{>0}$, and, for each $1\leq i\leq l$, a function $\alpha_i \, :\, \GH \to \dr R_{>0}$ such that the orbit of the sets $\G_t=G_t\cap \G$ for $t\in I_i$ becomes distributed in $\GH$ according to the density $\alpha_i$ with respect to $m_\GH$. That means, for all $\psi\in \mathcal C_c(\GH)$, we have:
$$\frac{1}{m_H(H_t)}\sum_{\g\in\G_t}\psi(\tau(\g)) \xrightarrow[t\in I_i]{t\to+\infty} \int_\GH \psi(x)\alpha_i(x)dm_\GH(x)\;.$$
\end{theorem}

The partition of the parameter space in a finite number of subspaces is not needed when there is no non-archimedean places as in \cite{goroweiss} but appears  even with very simple examples as soon as ultrametric part is to be taken in consideration. Let us also precise that the densities $\alpha_i$ are explicitely described and effectively computable in examples given afterwards (see theorem \ref{the:duality}).

We present here some examples of applications. Of course one may look at numerous situations. I just present here some variations about linear actions of the special linear group on points or subspaces. I believe that these examples show how to apply the previous theorem to specific situations, using algebraic features such as strong approximation in the special linear group. The proofs are postponed to section \ref{sec:exa}.

\subsubsection{Applications to $\d{SL}(2)$}

Consider the group $G=\d{SL}(2,\dr R)\times \d{SL}(2,\dr Q_p)$ for $p$ a prime number, and fix the lattice $\G=\d{SL}(2,\dr Z[\frac{1}{p}])$. We fix here (for sake of simplicity) the standard euclidean norm $| . |_\infty$ on the matrix algebra $\mathcal M(2,\dr R)$ and the max-norm $|.|_p$ on $\mathcal M(2,\dr Q_p)$. For a point $v$ in $\dr R^2$, we note also $|v|_\infty$ the norm of the matrix whose first column is $v$ and the second one is $0$. We define similarly the norm of a point in $\dr Q_p^2$. We choose a Haar measure $m=m_\infty\otimes m_p$ on $G$.

First we look at the action on the real plane, proving a result similar to Ledrappier's theorem but for the action of matrices in $\G$ subject to congruence conditions on their coefficients modulo $p$:

\begin{appl}\label{appl21}
Let $O$ be a bounded open subset of $\d{SL}(2,\dr Q_p)$. Note $\G^O_T$ the set of elements $\g\in \G$ such that $|\g|_\infty \leq T$ and $\g\in O$ as an element of $\d{SL}(2,\dr Q_p)$. Let $v$ be a point of the plane  $\dr R^2\setminus \{0\}$ with coordinates independant over $\dr Q$.

Then we have the following limit, for any function $\phi$ continuous with compact support in $\dr R^2\setminus \{0\}$:
$$ \frac{1}{T} \sum_{\G_T^O} \phi(\g(v)) \xrightarrow{T\to \infty} \frac{m_p(O)}{m(\GG)|v|_\infty}\int_{\dr R^2} \phi(w)\frac{dw}{|w|_\infty}$$\end{appl}

Another action of $\G$ of interest is on the product of real and $p$-adic planes. A precision : on the $p$-adic plane, we normalize the measure such that it gives mass $1$ to $\dr Z_p^2$. The result is that if your beginning point generates the whole plane among the $\dr Q$-subspaces, then its orbit is dense and you get a distribution result (the function $E$ appearing is the integer part):

\begin{appl}\label{appl22}
Let $(v_\infty,v_p)$ be an element of $(\dr R^2\setminus{0})\times(\dr Q_p^2\setminus{0})$. Suppose that any $\dr Q$-subspace $V$ of $\dr Q^2$ verifying $v_\infty\in V\otimes_{\dr Q} \dr R$ and $v_p \in V\otimes_{\dr Q} \dr Q_p$ is $\dr Q^2$. Denote $\G_T$ the set of elements $\g\in \G$ with $|\g|_\infty\leq T$ and $|\g|_p\leq T$.

Then, for all function $\phi$ continuous with compact support in $(\dr R^2\setminus{0})\times(\dr Q_p^2\setminus{0})$, we have the following limit:
$$\frac{1}{T p^{E(\ln_p(T))}} \sum_{\G_T} \phi(\g v_\infty, \g v_p) \xrightarrow{T\to\infty} \frac{p^2-1}{p^2 m(\GG)|v_\infty|_\infty |v_p|_p} \int_{\dr R^2\times \dr Q_p^2} \phi(v,w) \frac{dv dw}{|w|_\infty |w|_p}$$
\end{appl}

All these results may be extended with the tools presented in the paper for any norm on the matrix algebras and  by considering not only a prime number but a finite number of them.

\subsubsection{Applications to $\d{SL}(n)$}

We look here at a generalization in greater dimension. We consider the action of $\G=\d{SL}(n,\dr Z)$ on the $k$-th exterior power $\Lambda^k(\dr R^n)$, or the space of $k$-planes equipped with a volume. Once again we fix the standard euclidean norm $|.|$ on $\mathcal M(n,\dr R)$, but this time it is necessary to apply our theorem (see section \ref{sec:exa}). We consider also the standard euclidean norm  $|.|$ on $\Lambda^k(\dr R^n)$. And $m$ is a Haar measure on $\d{SL}(n,\dr R)$.
We get:

\begin{appl}\label{appln}
Let $v$ be a non-zero element of $\Lambda^k(\dr R^n)$ such that its corresponding $k$-plane of $\dr R^n$ contains no rational vector. Denote $\G_T$ the set of elements $\g\in \G$ with $|\g|\leq T$.

Then we have a positive real constant $c$ (independant of $\G$ and $v$) such that for all function $\phi$ continuous with compact support on $\Lambda^k(\dr R^n)\setminus \{0\}$:
$$\frac{1}{T^{n^2+k^2-nk-n}}\sum_{\G_T} \phi(\g v) \xrightarrow{T\to\infty} \frac{c}{m(\GG)|v|}\int_{\Lambda^k(\dr R^n)} \phi(v') \frac{dv'}{|v'|}$$
\end{appl}

The $S$-arithmetic generalization of the previous result holds of course. I prefer to postpone its statement and its proof to the section \ref{sec:exa}. Moreover I do not want to multiply here applications but one may think at examples in special unitary groups or Spin groups instead of the special linear one.

\subsection{Organization of the paper}

The organization of the paper is the following : in the next section we work out the so-called duality phenomenon, reducing the stated theorem to two results : a statement on volume of balls in the group and an analog of a result of Shah about equidistribution of balls of $H$ in $\GG$. The third section is devoted to the study of volume of balls, using $p$-adic integration. In the fourth section we review some tools we need to prove the analog of Shah theorem : mainly Ratner theorem for unipotent flows in a $p$-adic setting and several results due to G. Tomanov for polynomial dynamics in $S$-arithmetic homogeneous spaces. The fifth section is the devoted to some technical work. We conclude the proof in the sixth section. Eventually we treat the examples in the last section.

\tableofcontents

\section{Duality}

The duality phenomenon, as used by F. Ledrappier \cite{ledrappier} and A. Gorodnik-B. Weiss \cite{goroweiss}, is a consequence of the following idea : a property of the action of $\G$ on $\GH$ reflects in a property of the action of $H$ on $\GG$. The simplest example is the density of an orbit : $Hg$ has dense orbit under $\G$ in $\GH$ if and only if $g\G$ has dense orbit under $H$ in $\GG$. This consideration leads to the key point in the proof of Ledrappier : instead of looking at the orbit of the lattice $\G$ in the space $\GH$, we prefer to translate the problem in terms of the action of $H$ in $\GG$. And then we may use the precise description of unipotent orbits in the space $\GG$, namely Ratner's theory (cf section \ref{sec:ratner}) to prove some equidistribution results.
However, for asymptotic distribution of points, this phenomenon is not granted and requires additional assumptions we will review in this section.

We may remark that if $H$ is symmetric, Y. Benoist and H. Oh used other techniques - i.e. the mixing property - to study asymptotic distribution of orbits \cite{benoist-oh}.

In \cite[Corollary 2.4]{goroweiss}, Gorodnik and Weiss  presented an axiomatic frame for duality. Unfortunately we cannot use directly their statement as we miss some continuity hypothesis on the distance function - once again the ultrametric part has to be handled specifically, even if the final result holds. So we present a slightly adapted version of their result in the theorem \ref{the:duality}.

In the setting defined in the precedent section, consider an increasing and exhausting family $G_t$ of open bounded subsets in $G$. We need an hypothesis of regularity on this family. We choose to state it using the right action of open subsets of $G$ and asking the sets $G_t$ to be uniformly almost invariant by some open set. As we are interested in the intersections with $H$, the precise (and classical) definition is: 

\begin{definition}
Let $(G_t)_{t\in I}$ be a family of open bounded subsets of $G$. We say that it is \emph{almost (right)-invariant} if for every $\epsilon>0$ one can find on open neighborhood $U_\epsilon$ of $id$ in $G$ such that the two following inequalities hold for every $t \in I$:
\begin{itemize}
\item the set $G_t U_\epsilon$ is not too big with respect to $G_t$ inside $H$:
$$m_H(H\cap G_t U_\epsilon\setminus G_t) \leq \epsilon m_H(H\cap G_t)\; ,$$
\item Not too much points inside $G_t$ are $U_\epsilon$-closed to its complement inside $H$ : $$m_H(H\cap G_t \setminus G_t^cU_\epsilon)\geq (1-\epsilon)m_H(H\cap G_t)\;.$$
\end{itemize} 
\end{definition}
One easily checks that the balls $G_t$ defined by a size function on $G$ are almost invariant. Indeed for the archimedean part, any norm on the matrix algebra is continuous. And for the ultrametric part the max-norm is invariant under some open neighborhood of identity.

We also need a result of existence of limits for ratios of volumes of skew-balls in $H$ (Hypothesis D2 in \cite{goroweiss}). Recall the definition \ref{def:skewball} : for $g\in G$ and $t \in I$, $H_t(g)$ is the set $H\cap G_tg^{-1}$.
\begin{definition}
We say that a family $(G_t)_{t\in I}$ admits \emph{volume ratio limits for $H$} if for all $g$ in $G$ the ratio $\frac{m_H(H_t(g))}{m_H(H_t)}$ admits a limit as $t$ goes to $+\infty$ in $I$.
\end{definition}

The corollary 2.4 of \cite{goroweiss} (and its proof) implies the following one :
\begin{theorem}\label{the:duality}
Let $(G,H,\G)$ be a triple under study. Let $(G_t)_{t\in I}$ be a family of bounded open subsets of $G$ almost invariant, admitting volume ratio limits for $H$ and such that the volumes of $H_t=H\cap G_t$ go to $+\infty$. Assume moreover that the orbit of $H_t$ in $\GG$ becomes equidistributed with respect to $m_\GG$ ; i.e. for all $\phi\in \mathcal C_c(\GG)$, we have: $$\frac{1}{m_H(H_t)}\int_{H_t}\phi(\pi(h))dm_H(h)\xrightarrow[t \in I]{t\to +\infty} \int_\GG\phi dm_\GG\;.$$

Then the orbit of $\G_t=G_t\cap \G$ is distributed in $\GH$ according to a density with respect to $m_\GH$ ; i.e. for all $\psi\in \mathcal C_c(\GH)$, we have: $$\frac{1}{m_H(H_t)}\sum_{\g\in\G_t}\psi(\tau(\g)) \xrightarrow[t \in I]{t\to+\infty} \int_\GH \psi(H.g)\frac{m_H(H_T g)}{m_H(H_T)}dm_\GH(Hg)\;.$$
\end{theorem}

Particularly the density of the limit measure is described as limit ratio of volumes of balls. We will see in the next section a proof of existence of these ratios. But we will not in this paper go into precise and general estimates of these volumes. Our theorem still benefits of these estimation when available, e.g. in the applications (see section \ref{sec:exa}). Maucourant \cite{maucourant} get very precise estimations for $H$ real semisimple.

\begin{proof}
The proof is the same as \cite[Part 3 and 4]{goroweiss}: the almost invariance replacing the hypothesis of right continuity of the distance function in their paper.
\end{proof}

Now we have to understand the right setting to apply this theorem. There are two difficulties : the existence of volume ratio limits and the equidistribution of $H$-orbits in $\GG$. The next section address the first problem. We will prove the following theorem:
\begin{theorem}\label{the:volumeratiolimits}
Let $(G,H,\G)$ be a triple under study, $D$ a size function on $G$. Consider $(G_t)_{t\in \dr R}$ the family of balls for $F$. Suppose that the volume of $H_t$ goes to $+\infty$. 

Then there exists a finite partition of $\dr R$ in unbounded subsets $I_1,\ldots,\,I_k$ such that for all $1\leq l\leq k$ the family $(G_t)_{t\in I_l}$ admits volume ratio limits for $H$.
\end{theorem}
We shall exhibit in the following section a very simple example showing that we really need this partition.

The second part of the paper is to prove the equidistribution property under the hypothesis of theorem \ref{the:normball}: $H$ is a semidirect product of a semisimple and a unipotent groups and every dominant subgroup has dense orbit in $\GG$. We will prove in section \ref{section:equidistribution} the following theorem:
\begin{theorem}\label{the:Horbit}
Let $(G,H,\G)$ be a triple under study, $D$ a size function and $H_t$ the induced family of balls in $H$. Assume that every dominant subgroup $H'$ of $H$ has dense orbit in $\GG$.

Then the orbits of $H_t$ becomes equidistributed in $\GG$ with respect to $m_\GG$ ; i.e. for all $\phi\in \mathcal C_c(\GG)$, we have: $$\frac{1}{m_H(H_t)}\int_{H_t}\phi(\pi(h))dm_H(h)\xrightarrow{t\to +\infty} \int_\GG\phi dm_\GG\;.$$
\end{theorem}

Theorem \ref{the:normball} is then a direct consequence of the three previous results.

\section{Asymptotic developments of volumes}

\subsection{An example}

The following part is a little bit technical and may be misunderstood without any example in mind. Let us show on a very simple example that we have to be careful in describing the asymptotics of volumes of balls.

We will take here $G=\d{SL}(3,\dr R)\times \d{SL}(3,\dr Q_p)$ for some prime $p$ and $H$ the image under the adjoint representation of $\d{SL}(2)$ of the upper triangular nilpotent subgroup:
$$H=\left\{h(t_\infty,t_p)=\left( \begin{pmatrix} 1 & 2t_\infty &t_\infty^2 \\ 0& 1 &t_\infty \\ 0&0&1  \end{pmatrix}, \begin{pmatrix} 1 & 2t_p &t_p^2 \\ 0& 1 &t_p \\ 0&0&1  \end{pmatrix}\right) \; ; \; t_\infty \in \dr R \d{ and } t_p \in\dr Q_p\right\}$$

We choose the max-norm on both $\mathcal M_3(\dr R)$ and $\mathcal M_3(\dr Q_p)$ such that: $$H_{p^n}=\left\{h(s_\infty,s_p) \d{ for } s_\infty \in \dr R\d{ with }|s_\infty^2|\leq p^n\d{ and }s_p \in \dr Q_p\d{ with }|s_p^2|_p\leq p^n\right\}\; .$$
Hence the volume of $H_{p^n}$ is equal to $p^{\frac{n}{2}+E(\frac{n}{2})}$ ($E$ is the \emph{integer part}).

Now let us have a look on a specific skew-ball : $H_{p^n}(Id,\begin{pmatrix} p & 0 & 0\\ 0&1&0\\0&0&p^{-1}\end{pmatrix})$, and we note $g=(Id,\begin{pmatrix} p & 0 & 0\\ 0&1&0\\0&0&p^{-1}\end{pmatrix})$. Then the skew-ball is described by: $$H_{p^n}(g)=\left\{ h(s_\infty,s_p) \d{ for } |s_\infty|^2\leq p^n \d{ and }|p^{-1} s_p^2|_p \leq p^n \right\}\; ,$$ hence its volume $m_H(H_{p^n}(g))$ is equal to $p^{\frac{n}{2}+E(\frac{n-1}{2})}$.
We see that the ratio $\frac{m_H(H_{p^n}(g))}{m_H(H_{p^n})}$ is equal to $p^{E(\frac{n}{2})-E(\frac{n-1}{2})}$. This sequence does not admit any limit as $n$ goes to $\infty$. But we can split it in two subsequences: $n$ odd or even. And then both subsequences admit a limit (respectively $p$ and $1$).

Keeping this example in mind we will now explain why we are always able to do this: split the space of parameters $t$ in a finite number of subspaces in which the hypothesis of admitting volume ratio limits is fulfilled.

\subsection{Volume ratio limits}

We will prove here the theorem \ref{the:volumeratiolimits} stated above. We will use the fact that if two functions have an asymptotic development on the same (reasonnable) scale and their ratio is bounded, then this ratio admits a limit.

In order to get this asymptotic behaviour, we use the algebraic hypothesis on the norm. Then, following Benoist-Oh \cite[Part 16]{benoist-oh}, we get the wanted result as consequence of resolution of singularities in the archimedean case and Denef's Cell decompostion theorem in the non-archimedean one. These result are the two following propositions:

\begin{proposition}[Benoist-Oh, \cite{benoist-oh} Proposition 7.2]\label{pro:benoist-oh}
Let $H$ be the group of $\dr R$-points of an algebraic $\dr R$-group, $\rho\; :\; H \to GL(V)$ a $\dr R$-representation of $H$, $m_H$ the Haar measure on $H$ and $|\,.\,|$ an algebraic norm on $\d{End}(V)$.

Then, for all $g\in \d{GL}(V)$, the volume $m_H(H_t(g))=m_H\{h\in H\, |\rho (h)g|\leq t\}$ has an asymptotic development on the scale $t^a ln(t)^b$ with $a\in \dr Q^+$ and $b\in \dr N$.
\end{proposition}

For the ultrametric part, we do not get exactly an asymptotic development rather a finite number of asymptotic developments. This was already noted in \cite{benoist-oh} but we need here a slightly more precise result, namely a uniformity on the number of simple functions needed: 
\begin{proposition}[Benoist-Oh]\label{the:benoist-oh}
Let $k$ be a finite extension of $\dr Q_p$, $q$ be the norm of an uniformizer, $H$ the group of $k$-points of an algebraic $k$-group, $\rho\; :\; H \to GL(V)$a $k$-representation of $H$, $m_H$ the Haar measure on $H$ and $|\,.\,|$ a $\d{max}$-norm on $\d{End}(V)$. Let $S_t(g)$ be the sphere of radius $t$ : $S_t(g):=\{h\in H\st |hg|=t\}$.

Then there exist $N_0$ an integer such that for all $g\in G$ and for each $0\leq j_0\leq N_0$ one of the following holds:
\begin{enumerate}
\item $S_{q^j}(g)$ is empty for all $j=j_0 \d{ mod }N_0$.
\item There exist $d_{j_0}\in \dr Q_{\geq 0}$, $e_{j_0}$ an integer and $c_{j_0}>0$ such that $m_H(S_{q^j}(g))\sim c_{j_0} q^{d_{j_0} j} j^{e_{j_0}}$ for all $j=j_0 \d{ mod }N_0$.
\end{enumerate}
\end{proposition}

\begin{proof}
I will not go into details as the proof is  the same as \cite[Corollary 16.7]{benoist-oh}. I will just say that applying a theorem of Denef \cite[Theorem 3.1 and remark below]{denef}, we get the following: 

for any polynomial map $f(x,\lambda)$ from $\dr Q_p^{m+d}$ to some $GL(V)$, for any semialgebraic measure $\mu$ on a semialgebraic set $S\subset \dr Q_p^m$, there are some functions $\gamma_i(\lambda,n)$ and $\beta_i(\lambda,n)$ for $1\leq i\leq e$ such that  the measure $I(\lambda,n)$  of the set of element $x\in S$ with $|f(x,\lambda)|=q^n$ is of the form :
$$I(\lambda,n)=\sum_{i=1}^e \gamma_i(\lambda,n)p^{\beta_i(\lambda,n)}$$

Moreover the functions $\gamma_i$ and $\beta_i$ are simple in the following sense: for any of these functions (hereafter denoted $\alpha$) there exists an integer $N$ such that for all $\lambda$, the map $n\mapsto \alpha(\lambda,n)$ is affine along at most $N$ arithmetic progressions in $\dr N$ which cover $\dr N$ up to a finite set.

\smallskip

Now, the above proposition is just this result in the case where $S$ is the image under the representation $\rho$ of $H$, $\mu$ is the Haar measure on $H$ and $f(\lambda,x)=\lambda.x$ for $\lambda\in GL(V)$ and $x\in H$.
\end{proof}
\smallskip

We may go on with the proof of theorem \ref{the:volumeratiolimits}.
Let us write more explicitly the informations we get on the function $m_H(H_t(g))$ from this two results. Fix some $g$ in $G$. Consider the set  $S_f$ of finite places in $S$. For each $\nu \in S_f$, we note $q_\nu$ the norm of the uniformiser of $K_\nu$. The previous proposition gives us an integer $N_\nu$ and for all $0\leq j\leq N_\nu-1$ some $d_{\nu,j}\in \dr Q$, $d_{\nu,j}\in \dr N$ and $c_{\nu,j}>0$ describing the volume of spheres in the group $H_\nu$. Moreover for the archimedean part, the proposition \ref{pro:benoist-oh} gives some triple $d_\infty \in \dr Q_{>0}$, $e_\infty \in \dr N$ and $c_\infty >0$ such that the volume of $(H_\infty)_t$ is equivalent to $c_\infty t^{e_\infty} e^{d_\infty t}$. With this data we are able to describe the volume of $H_t$:
\begin{lemma}\label{equivalentvolume}
With the data above, $m_H(H_t(g))$ is equivalent, as $t$ goes to $\infty$, to : 
\begin{eqnarray}\label{for:volume}
c_\infty t^{d_\infty}(\ln\, t)^{E_\infty}\prod_{\nu \in S_f} \left(\sum_{j=0}^{E(\ln_{q_\nu}t)} c_{\nu, j[N_\nu]}q_\nu^{d_{\nu, j[N_\nu]} j} j^{e_{\nu, j[N_\nu]}} \right)\; .
\end{eqnarray}

Moreover, let $\displaystyle d=d_\infty \times \prod_{\nu\in S_f} \d{max}_{0\leq j\leq N_\nu} d_{\nu,j}$ and  $\displaystyle e=e_\infty \times \prod_{\nu\in S_f} \d{max}_{0\leq j\leq N_\nu} e_{\nu,j}$. Then $m_H(H_t(g))$ lies between two constants times $t^e e^{dt}$.
\end{lemma}

\begin{proof}
By definition of the size function, the ball $H_t(g)$ is the product for all $\nu$ in $S$ of the balls $(H_\nu)_t(g_\nu)$ in the group $H_\nu$. For each of these balls the two previous theorems give us an equivalent for the volume in $H_\nu$ (all functions are positive so there is no trouble summing equivalent). Now the Haar measure on $H$ is the product of the Haar measures on the $H_\nu$'s. And the formula of the previous theorem is just the product of these equivalences.

The second part directly comes from the first one.
\end{proof}

The following lemma is the last step:
\begin{lemma}
Under the hypothesis of theorem \ref{the:volumeratiolimits} fix an element $g$ in $G$.

Then there exists a constant $c>1$ such that the ratio $\frac{m_H(H_t(g))}{m_H(H_t)}$ lies between $c^{-1}$ and $c$ for all $t$.
\end{lemma}
\begin{proof}
The element $g$ acts continuously on the module $\d{End}({\bf V}(K_S))$ (recall that in order to define balls in $G$ we fixed some representation of $\bG$ in a vector space ${\bf V}$). So there are two constants $A$ and $B$ such that we have for all $h$ in $H$ (recall that $D$ denotes the size function) :
$$A.D(h)\leq D(hg) \leq B. D(h)$$
That implies that the set $H_t(g)$ contains $H_{At}$ and is contained in $H_{Bt}$.

But the second part of the previous lemma implies that  the ratios $\frac{m_H(H_{At})}{m_H(H_t)}$ and $\frac{m_H(H_{Bt})}{m_H(H_t)}$ are bounded. Hence we have proven the lemma.
\end{proof}

We now have the tools to proceed with the proof of theorem \ref{the:volumeratiolimits}:
\begin{proof}
Each finite place leads to a finite partition of the space of parameters in the following way: For $\nu \in S_f$ we have $q_\nu$ the norm of the uniformizer and the integer $N_\nu$ given by the theorem \ref{the:benoist-oh}. For $0\leq j\leq N_\nu-1$ we call $I_{\nu,j}$ the set of real numbers $t$ such that $E(\ln_{q_\nu}t)$ is equal to $j$ modulo $N_\nu$. The theorem \ref{the:benoist-oh} implies that \emph{on the sets} $I_{\nu,j}$ and for all $g\in G$ we have a asymptotic development of the volume of $(H_\nu)_t(g)$ of the form: $m_H((H_\nu)_t(g)) \sim  C_{\nu,j}t^{E_{\nu,j}}e^{D_{\nu,j} t}$.

Now consider the finite partition $I_1,\ldots,I_l$ of $\dr R$ given by the intersection of all these partitions. Then on a set $I_j$ of this partition and for all $g$ in $G$, the volume $m_H(H_t(g))$ is equivalent to some $C_j(g) t^{E_j(g)} e^{D_j(g)t}$. But we know by the previous lemma that the ratio $\frac{m_H(H_t(g))}{m_H(H_t)}$ is bounded.

At this point we are done: since the ratio is bounded, we have $E_j(g)=E_j(Id)$ and $D_j(g)=D_j(Id)$. Hence the ratio admits a limit (depending on the set $I_j$), namely $\frac{C_j(g)}{C_j(Id)}$.
\end{proof}

\section{Polynomial dynamic in homogeneous spaces}\label{sec:ratner}
We here recall some facts about polynomial dynamic in $S$-arithmetic groups. The result we need can mainly be found in Tomanov \cite{Tomanov1}. They are also used in \cite{gorodnik-oh}. The main difference here - which is only a technical one - is that we need to extend all the results to orbit of polynomial in several variables. This does not change deeply the proof of the theorems. The interested reader may refer to the author's PhD thesis \cite{mathese} for details.

\subsection{Measure on $\GG$ invariant under the action of a unipotent subgroup}

\subsubsection{Measure rigidity in an $S$-arithmetic setting}\label{sssec:ratner}

We need the rigidity theorem for measures invariant under an unipotent group, often called Ratner's theorem. For $p$-adic groups, it has been proved by Ratner and Margulis-Tomanov. But in an $S$-arithmetic setting a more precise version can be found in \cite{Tomanov1}.

Accordingly to \cite{Tomanov1}, we define the notion of subgroup of class $\mathcal F$ :

\begin{definition}
Let ${\bf A}$ be a $\dr Q$-subgroup of $\bG$. Then ${\bf A}$ belongs to the class $\mathcal F$ if and only if ${\bf A}(K_S)$ is the Zariski closure of the group generated by the unipotent elements of ${\bf A}(K_S)$.
\end{definition}

Recall from \cite{Tomanov1} that for a class $\mathcal F$-group ${\bf P}$, the subgroup $P\cap \G$ is a lattice in $P$. It implies that the projection of $P$ in $\GG$ is closed.

We can now state the measure rigidity theorem :

\begin{theorem}[Ratner, Margulis-Tomanov, Tomanov]\label{the:ratner}
Let $\bG$ be a $\dr Q$-group, $\G$ an arithmetic subgroup of $G={\bG}(K_S)$ and $U$ a subgroup of $G$ generated by its one-parameter unipotent subgroups.

Then for all probability measure $\mu$ on $\GG$ which is $U$-invariant and $U$-ergodic, there exist a class $\mathcal F$-subgroup ${\bf P}$ of $\bG$ and $P'$ a finite index subgroup of $P={\bf P}(K_S)$ such that the probability $\mu$ is the $P'$-invariant probability on a translate of a $P'$-orbit in $\GG$. 
\end{theorem}

This theorem allows a complete description of $U$-invariant probability measures.

\subsubsection{The non-ergodic case}

Let $U$ be a subgroup of $G$ generated by its one-parameter unipotent subgroups and $\mu$ be a $U$-invariant probability measure on $\GG$.

For each class $\mathcal F$ subgroup of $\bG$, the precedent theorem defines a class of $U$-ergodic probability measures. To understand the decomposition of $\mu$ into ergodic components, we have to define some subsets of $G$ :

\begin{definition}
Let ${\bf P}$ be a class $\mathcal F$ subgroup of $\bG$. Then the sets $X(P,U)$ and $S(P,U)$ are defined in the following way :
\begin{eqnarray*}
X(P,U) &=& \left\{g \in G \st Ug \subset gP\right\} \\
S(P,U) &=& \bigcup_{{\bf P}'\in
\mathcal F \; , \; {\bf P}'\subset {\bf P}} X(P',U)
\end{eqnarray*}
\end{definition}

We remark that $X(P,U)$ is an algebraic subvariety of $G$. 

For each class $\mathcal F$ subgroup ${\bf P}$ of $\bG$, let $\mu_{\bf P}$ be the restriction of $\mu$ to $\pi(X(P,U)-S(P,U))$. Then each ergodic component of $\mu_{\bf P}$ is of the form given by the precedent theorem for this group ${\bf P}$. Moreover, since the sets $\pi(X(P,U)-S(P,U))$ are disjoint we get the following decomposition of $\mu$ in a denombrable sum : $$\mu = \sum_{{\bf P}\in \mathcal F} \mu_{{\bf P}} \;.$$

This decomposition enlightens the following fact : in order to understand a measure $U$-invariant, we have to understand the behaviour of trajectories near the variety  $\pi(X(P,U)-S(P,U))$. The goal of  this section is to get a such a result. But first of all, we will define some useful representations of the group $G$.

\subsection{A suitable representation}

We fix here a class $\mathcal F$-subgroup ${\bf P}$. Chevalley's theorem \cite[5.1]{Borel} grants the existence of a $K$-representation $\rho_P$ of $\bG$ such that ${\bf P}$ is the stabilizer of a line ${\bf D}$ in the space ${\bf V}_P$ of the representation.

We fix a point $v_P$ in ${\bf D}(K)$. Moreover we consider $v_P$ as a point of the $K_S$-module $V_P={\bf V}_P(K_S)$. We now get a function $\eta_P$ from $G$ to $ V$ given by the following formula : $$\eta_P(g)=\rho_P(g).v_P \; .$$

The normalizer ${\bf N(P)}$ of ${\bf P}$ fix the line ${\bf D}$ but not the point $v_P$. So we define ${\bf N_1(P)}$ to be the fixator of the point $v_P$.

The following lemma will be useful, as a link between properties of subset in $\GG$ and in $V_P$ :

\begin{lemma}
\begin{itemize}
\item The set $\eta_P(\G)$ is discrete in $V_P$. 
\item The set $N_1(P)\G /\G$ is closed in $\GG$.
\end{itemize}
\end{lemma}

\begin{proof}
First the subgroup ${\bf V}_P(\mathcal O_S)$ is discrete in $V_P={\bf V}_P(K_S)$  and $\rho_P$ is a $K$-representation. So the set $\rho_P({\bf G}(\mathcal O_S)).v_P$ is discrete in $V_P$. Moreover $\G$ is supposed to be arithmetic, so $\eta_P(\G)$ is contained in a finite number of translates of $\rho_P({\bf G}(\mathcal O_S)).v_P$. Hence it is a discrete set.

Second, let $g_k=n_k \g_k$ be a sequence of points in  $N_1(P)\G$ and assume that $g_k$ converges to a point $g$. We want to prove that $g\G/\G$ belongs to $N_1(P)\G /\G$. We rewrite the definition of $g_k$ : $\g_k^{-1}=g_k^{-1}n_k$. By definition of $N_1(P)$, we then get $\eta_P(\g_k^{-1})=\eta_P(g_k^{-1})$. We just showed that $\eta_P(\G)$ is discrete. So the sequence $\g_k$ is stationary equal to a $\g$ for $k$ large enough. Then $g_k \g^{-1}$ fixes $v_P$ for $k$ large enough. That is $g_k \g^{-1}$ belongs to $N_1(P)$. So does its limit and we can conclude : $g$ belongs to $N_1(P)\G$.
\end{proof}

We conclude with a last definition involving the group $U$. The set $X(P,U)$ is $N(P)$-invariant hence $N_1(P)$-invariant by right multiplication and it is a Zariski closed set of $G$. So its image by the function $\eta_P$, which is Zariski-open and surjective on $\eta_P(G)$, is Zariski-closed in $\eta_P(G)$. However there is no reason for it to be Zariski-closed as well in $V_P$. So we define $F(P,U)$ as the Zariski-closure of $\eta_P(X(P,U))$ in $V_P$.

\begin{remark}
To avoid confusion, let us describe the Zariski topology in $K_S$-modules : a polynomial $Q$ of $K_S[X_1,\ldots,X_n]$ is nothing else than a collection of polynomial $Q_\nu$ for all $\nu$ in $S$. A Zariski-closed subset of a $K_S$-module $\displaystyle M=\prod_{\nu \in S} m_\nu$  is then naturally an intersection of products of Zariski-closed subsets of each $M_\nu$
\end{remark}

\subsection{Behavior of polynomial functions}
We now state a theorem allowing to control polynomial dynamics along the sets
$\pi(X(P,U) - S(P,U))$. Let us begin by the definition of a polynomial function in the $K_S$-points $G$ of a $K$-group $\bG$ with a faithful linear representation $\rho$: a function
$f=(f_\nu)_{\nu\in S}$ from $(K_S)^m$ to $G$ is said polynomial of degree $d$ if for all $\nu \in S$, the matrix entries of $\rho\circ f_\nu$  are all polynomial of degree $d$. The set of functions from $K_S^m$ to $G$ polynomial of degree at most $d$ will be noted $\mathcal P_{d,m} (G)$. Moreover we note $\theta=\bigotimes_{\nu in S} \theta_\nu$ the Haar measure on $K_S$ normalized such that the volume of $K_S/\mathcal O_S$ equals $1$ and $\theta_m=\bigotimes^m \theta$ the induced measure on $K_S^m$.

Recall the definition of $\eta$ from $G$ to some $K$-module $V_G$ given by Chevalley's theorem. Moreover $F(P,U)$ has been defined as the Zariski closure of $\eta(X(P,U))$ inside $V_G$. Hereafter, we call \emph{cube} in $(K_S)^m$ a product of balls $\prod_{i=1}^m\prod_{\nu\in S} B_{i,\nu}$.

\begin{theorem}[Tomanov]\label{the:DM}
Let ${\bf G}$ be a $K$-group, $\G$ an arithmetic subgroup of
 $G= {\bf G}(K_S)$,  $U$ be a subgroup of $G$ generated by its one-parameter unipotent subgroups and ${\bf P}$ a class $\mathcal F$-subgroup. Let $C$ be a compact subset of $X(P,U) \G / \G$, $d$ and $m$ two integers and $\varepsilon >0$.

Then there exists a compact subset $D$ of $F(P,U)$ such that for all relatively compact neighbourhood $W_0$ of $D$ in $V_G$, there exists a neighbourhood $W$ of $C$ in
$\GG$,  such that for all $m$, for all cube $B$ in $(K_S)^m$, and all function
$f$ in $\mathcal P_{(d,m)} (G)$ we have :
\begin{itemize}
\item either we can find $\g$ in $\G$ such that $\eta(f(B)\g) \subset
W_0$
\item or $\theta_m(\left\{t \in B \st (f(t)\G / \G) \in W \right\}) < \varepsilon
\theta_m(B)$.
\end{itemize}
\end{theorem}

In \cite{Tomanov1} the theorem was not stated for functions in $\mathcal P_{d,m}$ but for one parameters unipotent orbits. However there is no conceptual jump in the proof of the above theorem. Moreover the real cases of this theorem (and of all this section) is well known \cite{shah}. The interested reader may find more technical details in the author's PhD thesis \cite{mathese}.

\subsection{Non-divergence of polynomial orbits}

We need a last result in order to control the divergence of polynomial orbit. The following theorem is a kind of analog of a result of Eskin-Margulis-Shah \cite{EMS}. However, we won't need the whole precision of their result, we may just use a slight adaptation of \cite[Theorem 8.4 and 9.1]{kleinbock-tomanov} :

\begin{theorem}[Kleinbock-Tomanov]\label{the:nondiv}
Let ${\bf G}$ be a $K$-group, $\G$ an arithmetic subgroup of
 $G= {\bf G}(K_S)$.  Fix $d$ and $m$ two integers.
 
Then there are a finite number of parabolic subgroups ${\bf P_k}$ of ${\bf G}$ and their associated Chevalley representations $\rho_k$ in a space $V_k$ with a marked point $v_k \in V_k$ in a line stabilized by ${\bf P_k}$ such that  :

for all $\varepsilon>0$ there are a compact $D$ in $\GG$ and  compact subsets $D_k$ in each $V_k$ verifying:
for all $f$ in $\mathcal P_{(d,m)} (G)$, for all cube $B$ in $(K_S)^m$, one of the following holds :
\begin{enumerate}
\item $ \theta_m(\left\{t \in B \st (f(t)\G / \G) \not\in  D \right\}) < \varepsilon \theta_m(B)$.
\item There is an integer $k$ such that there exists $\gamma \in \G$ with : $\rho_k(f(B)\gamma).v_k \subset D_k$.  
\end{enumerate}
\end{theorem}

\section{Some tools: Cartan decomposition, decomposition of measures and representations}

Our proof of theorem \ref{the:Horbit} requires some technical tools. The first one is more than classical: the Cartan decomposition in the semisimple part, which we recall to settle some notations. The second one is merely a way to note all the measures (and their translates) we will consider in the sequel, together with some basic lemmas. The third and last one is a lemma on representations of $H$. It is an extension of \cite[Part 5]{shah} to our setting.

\subsection{Cartan Decomposition in $H^{ss}$}

The group $H$ is a semidirect product of a semisimple part $H^{ss}$ and a unipotent one $H^u$. For the semisimple part we have a Cartan decomposition: for all $\nu$ in $S$ such that $H_\nu$ is non-compact we choose a maximal $K_\nu$-split torus $A_\nu$ in $H_\nu$. We choose then a system of positive simple restricted roots $\Phi_\nu$ thus defining the associated sub-semigroup $A^+_\nu$ of $A_\nu$. Then there exists maximal compact subgroups $C_\nu$ and finite sets $D_\nu$ in the normalizer of $A_\nu$ such that the following Cartan decomposition holds: $H_\nu$ is the disjoint union of the double class $C_\nu d a  C_\nu$ for $a\in A^+_\nu$ and $d\in D_\nu$. For the existence of these objects we refer to \cite{Tits}.
When $H_\nu$ is compact we just choose $C_\nu=H_\nu$, $A_\nu$ and $D_\nu$ are reduced to the identity.

Let $A^+=\prod_{\nu\in S} A^+_\nu$ and similarly $C$ and $D$ are the products of the $C_\nu$'s and $D_\nu$'s. Let $\Phi$ be the union of the $\Phi_\nu$. For $\alpha \in \Phi_\nu \subset \Phi$ and $a=(a_\nu)_{\nu \in S}$ we define $\alpha(a)=\alpha(a_\nu)$.

Consider a sequence $a_n$ of elements of $A^+$.
\begin{definition}
A sequence $a_n$ of elements of $A^+$ is \emph{simplified} if for all $\alpha$ in $\Phi$ we have the alternative:
\begin{itemize}
\item either $\alpha(a_n)$ is bounded,
\item or $\alpha(a_n)$ goes to $+\infty$
\end{itemize}

Associated to such a simplified sequence, we consider the contracted unipotent subgroup of $H^ss$.
$$U^+ = \left\{ h\in H^{ss} \st \lim_{n\to +\infty}a_n^{-1} h a_n =e \right\}\, .$$
\end{definition}
\begin{remark}
We did not assume that a simplified sequence $a_n$ is unbounded. So the group $U^+$ associated may be equal to the trivial group.
\end{remark}

\subsection{Decomposition of measures}

The idea is simple: given some measure $\mu$ on the ball $(H^{ss})_t$, we want to define a probability measure on the ball $H_t$ which disintegrates (in the product $H=H^{ss}\rtimes H^u$) on $\mu$ and the Haar measure in the fibers. The notations may seem tedious as we must work at each place in parallel. But it will proove useful later.

The assumptions made on the norm ensure the following : for all $h^{ss}$ in $H_\nu^{ss}$, the set of elements in $H_\nu^u$ such that $h^{ss}h^u$ belongs to $(H_\nu)_t$ is a ball of radius some $l_{[\nu,t]}(h^{ss}$) in $H_\nu^u$ and moreover depends continuously on $h^{ss}$ and $t$. So for all $t$, there is a continuous function $l_{[\nu,t]}$ from $H_\nu^{ss}$ to $\mathbb R^+$ such that :
$$(H_\nu)_t=\bigcup_{h\in H_\nu^{ss}} \{h\}\times (H_\nu^u)_{l_{[\nu,t]}(h^{ss})}$$
This in turn translates in terms of measures. We note $m^u_\nu(l)$ the restriction of the Haar measure $m_{H_\nu^u}$ to the ball $(H_\nu^u)_l$. And for measure $\mu_\nu$ in $H_\nu^{ss}$, we may define the measure $m_\nu(\mu_\nu,t)$ by the formula, for all $\phi$ continuous with compact support on $H_\nu$ :

$$\int_{H_\nu} \phi dm_\nu(\mu_\nu,t)=\int_{H^{ss}_\nu} \int_{(H_\nu^u)_{l_{[\nu,t]}(o)}} \phi(ob)dm^u_\nu(l_{[\nu,t]}(o))(b)d\mu_\nu(o)$$

For $\mu=\bigotimes \mu_\nu$ a product measure on $H^{ss}$ of finite total mass and $t$ positive, we note $m(\mu,t)$ the product $\bigotimes_{\nu \in S} m_\nu(\mu_\nu,t)$'s. Eventually we note $\dr P(\mu,t)$ the renormalized probability measure and $\d{Supp}(\mu,t)$ its support. Remark that, if $\mu$ proportionnal to the Haar measure of some subgroup $S$ in $H^{ss}$, then $m(\mu,t)$ is proportional to the Haar measure in $S\rtimes H^{u}$ restricted to $(S\rtimes H^u)_t$.

Let us immediatly state two lemmas showing that this probability measures behave well with respect to $\mu$ as soon as the support of $\mu$ does not approach the frontier of $H_t$. First look at translations:

\begin{lemma}\label{lem:limtranslate}
Let $\mu_n$ be a sequence of probability measure on $H^{ss}$ and $t_n$ go to $\infty$. Let $h_n$ go to $Id$ in $H^{ss}$. Assume that the support of $\mu_n$ is included in a ball of radius $H^{ss}_{(1-\varepsilon)t_n}$ for some $\varepsilon>0$.

Then the sequence of (signed) measure $\dr P(((h_n)_*\mu_n),t_n)-\dr P(\mu_n,t_n)$ converges to $0$.
\end{lemma}

\begin{proof}
The assumption on the supports of $\mu_n$ ensures that the supports of $(h_n)_*\mu_n$ are included in $(H^{ss})_{t_n}$ for $n$ big enough. Moreover (by left-uniform continuity of the norms) we have for every sequence $g_n$ in the support of $\mu$ and for all place $\nu$ (here I forget some indices $\nu$ to keep the formula readable): $$\frac{l_{[\nu,t_n]}(h_ng_n)}{l_{[\nu,t_n]}(g_n)}\xrightarrow{n\to \infty} 1\; .$$
As, eventually, the signed measures $(h_n)_*\mu_n-\mu_n$ go to $0$, the lemma is proven by a straightforward calculus.
\end{proof}

The second lemma allows to handle also a sequence of measure $\mu_n$:

\begin{lemma}\label{lem:limsupport}
Let $\mu_n$ be a sequence of probability measures on $H^{ss}$ converging to $\mu$ with all these measures supported in a given compact set and absolutely continuous with respect to some $\lambda$. Let $t_n$ be a sequence of real numbers going to $+\infty$ and $h_n$ a sequence of elements in $H^{ss}$.

Then the sequence of (signed) measure $\dr P(((h_n)_*\mu_n),t_n)-\dr P((h_n)_*\mu,t_n)$ converge to $0$.
\end{lemma}

\begin{proof}
By hypothesis, the signed measure $\mu_n-\mu$ has a density going to zero in $L^1(\lambda)$. But all these densities are supported inside a compact set. Hence $\mu_n-\mu$ has a total variation going to zero : for all $\epsilon$ there is $n$ such that for all function on $H^{ss}$, we get:
$$|\int f d\mu_n -\int f d \mu| \leq \epsilon \d{max}(|f|)$$

This ensures that its translates under $h_n$ go to zero i.e. that $\dr P((h_n)_*(\mu_n,t_n))-\dr P(((h_n)_*\mu),t_n)$ go to $0$.
\end{proof}

\subsection{A lemma on linear representation}

The first equidistribution result we will prove is for projections of probability measures of the form $\dr P((a_n)_* l,t_n)$ where $l$ is a probability measure on $U^+$ absolutely continuous with respect to the Haar measure. But we need a result on the action of the support $S((a_n)_* l,t_n)$ of this measure: it sends every non-invariant point to $\infty$. For technical reasons, we need this property directionnally in $H^u$, i.e. along $1$-parameter subgroups in $H^u$. 

The situation of this section is the following one: let $(a_n)$ be a simplified sequence. Let $\Omega$ be an open and relatively compact subset of $U^+$. Let $(t_n)$ be a sequence of real numbers going to $\infty$ such that the sets $a_n\Omega$ are included in balls $H^{ss}_{t_n}$. 
Let $N^{ss}$ be the smallest normal subgroup of $H$ such that the projection of $a_n$ is bounded in $H/N^{ss}$.

\begin{lemma}\label{lem:representation}
Let $\rho=(\rho_\nu)_{\nu\in S}$ be a $K_S$-representation of $H$ in a finite dimensional $K_S$-module $V=\prod V_\nu$. Let $O$ be a $1$-parameter subgroup of $H^u$, $O_n$ be the set $\{o\in \Omega\times O \st \d{for some }\omega\in \Omega\d{, }D(a_n o)\leq t_n\}$. Let $N_O$ be the smallest subgroup of $H$ such that $a_nO_n$ stay in a compact in $H/N_O$.

Let $\Lambda$ be a discret subset of $V$ with no $N_O$-invariant points and $v_n$ a sequence of elements of $\Lambda$.

Then the sequence of sets $\rho (a_n O_n) v_n$ is not contained in any compact subset of $V$.
\end{lemma}

This whole subsection will be the proof of this lemma.

\begin{proof}
We split this proof in two cases : whether the sequence $a_n$ is bounded or not. 

\emph{Case 1 : $a_n$ is bounded}

We may assume that all $a_n$ equal $1$.
Then $U^+$ is trivial, the $O_n$ is the ball $D(o)\leq t_n$ in $O$ and $N_O$ is  the group $O$. As $t_n$ go to $\infty$, we may extract a increasing subsequence of balls covering $O$. If $v_n$ is not bounded, as $Id$ belongs to $O_n$, then the lemma is proven. If not, as $\Lambda$ is discrete, we may assume that $v_n$ is constantly equal to some $v$ which is not $N_O$-invariant. Now the exponential function composed with $g\mapsto\rho(g)v$ gives us a polynomial function from the Lie algebra of $O$ to $V$, and $\rho(O)v$ is the image of this polynomial function. That means that this function is constant or unbounded. As it is not constant, it is unbounded, proving the lemma in this case.

\smallskip

\emph{Case 2 : $a_n$ is not bounded}

In this situation, the action of $a_n$ and $U^+$ alone send non-invariant points to $\infty$ (remark that $\Omega$ is included in $O_n$ by definition).

First of all, let $V^{N^{ss}}$ be the $N^{ss}$-invariant sub-module of $V$ and $W$ an $N^{ss}$-invariant complement. Write $v_n=v_n^{N^{ss}}+w_n$. If $w_n$ goes to $0$, by discreteness of $\Lambda$, $v_n^N$ goes to $\infty$. Let $C$ be a compact of $G$ such that $a_n O_n$ is included in $C N_O$. Then, by definition of $U^+$ and semisimplicity of $H^{ss}$, the sets $a_n U^+$ are included in $C N^{ss}$. And for any $\omega\in \Omega$, the sequence $\rho(a_n \omega) v_n=\rho(a_n \omega)(v_n^N)+\rho(a_n\omega)w_n$ belongs to $\rho(C)v_n^N+W$. Hence this sequence goes to $\infty$, proving the lemma in this case.

So we may assume that $w_n$ does not go to zero. Up to a renormalization and an extraction, we assume that $w_n$ converges to some non-zero element $w\in W$. It is enough to prove that the sets $\rho (a_n \Omega) w$ leave every compact of $V$. Making this reduction we loose the discreteness hypothesis on $\Lambda$ but we will not need it anymore.

\smallskip

We now prove the lemma by contradiction: suppose that the above sets stay in some compact. We prove first  that $w$ is $N^{ss}$-invariant and then $N$-invariant.

\smallskip

The first step is to show that we may assume that $w$ is $U^+$-invariant: let $V^+$ be the module of $U^+$-invariant points and $V^-$ its $a_n$-invariant complement. Note $p^+$ the projection on $V^+$ in the direction $V^-$. We have the following
\begin{lemma}
Let $\rho=(\rho_\nu)_{\nu\in S}$ be a $K_S$-representation of $H$ in a finite dimensional $K_S$-module $V=\prod V_\nu$. Let $U$ be a non-trivial unipotent subgroup, and $\Omega$ an open subset of $U$. 

Then the set $\rho(\Omega) w$ is not contained in any complement of the submodule $V^U$ of $U$-invariant points.
\end{lemma}
\begin{proof}
Once again we prove it by contradiction: suppose $\rho(\Omega) w$ generates some submodule $V'$ in direct sum with $V^U$. And let $\omega_1, \ldots , \,\omega_k$ be elements of $\Omega$ such that the $\rho(\omega_i)w$ generate $V'$. Then there is an neighborhood $\Omega'$ of the identity in $U$ such that all the $\Omega' \omega_i$ are included in $\Omega$.

And $V'$ is $\Omega'$ invariant. So it is invariant by the Zariski closure of $\Omega'$ i.e. by $U$. The Lie-Kolchin theorem implies that there is a non-zero $U$-invariant element in the $U$-invariant module $V'$ (to be very precise, you have to apply the Lie-Kolchin theorem at each place, restricting the representation in the obvious way). This is the contradiction: $V'$ cannot be in direct sum with $V^U$.
\end{proof}

So, there is some $\omega \in \Omega$ such that $p^+(\rho(\omega)w)$ is not zero. But we know that $\rho(a_n\omega)w$ is bounded. Hence $\rho(a_n) p^+(\rho(\omega)w)$ is bounded. Let us show that it implies that $N^{ss}$ is contained in the kernel of the representation:
\begin{lemma}
Let $v$ be a $U^+$-invariant and non-zero point of $V$ such that $\rho(a_n)v$ is bounded. Then $N^{ss}$ is contained in the kernel of the representation.
\end{lemma}
\begin{proof}
We may assume that at each place $\rho_\nu$ is an irreductible representation. 
First of all, let $W$ be the sub-$K_S$-module of $V$ containing all the vectors $w$ such that $\rho(a_n)w$ is bounded. Consider $P^-$ the opposite parabolic subgroup in $H^{ss}$:
$$P^-=\left\{h\in H\st a_n ha_n^{-1}\d{ remains bounded}\right\}\; .$$
Then it is clear that $\rho(P^-)v$ is included in $W$. By $U^+$ invariance of $v$, we even get that $\rho(P^-U^+)v$ is included in $W$. But $P^-U^+$ is open in $H$; so Zariski-dense. We deduce that $\rho(H)v$ is included in $W$ and by irreducibility that $W=V$.

Let us now prove that all the element of $V$ are $U^+$-invariant. We just have to prove it  on eigenvectors for the action of $a_n$ ($V$ is the sum of the eigenspaces for this action). Remind that, as  $a_n$ has determinant one and all the vectors have a bounded orbit under the action of $a_n$, all the eigenvalues of this action are of modulus $1$. So let $v'$ be in $V$ with $\rho(a_n)v'=\lambda_n v'$ and $\omega$ be some element of $U^+$. Fix an open neighborhood of the identity $\Omega$ in $U^+$. Then by definition there is some integer $i$ such that $a_i^{-1} \omega a_i$ belongs to $\Omega$. Hence $\rho(\omega)v'$ belongs to $\rho(a_i)(\rho(\Omega)\lambda_i^{-1}v')$. But the latter is included in some compact $B$ independent of $i$ because we have seen that all elements of $V$ have bounded orbit in $V$ and the sets $\rho(\Omega)\lambda_i^{-1}v'$ are contained in some compact.
So $\rho(U^+) v'$ is included in $B$. But $U^+$ is an unipotent subgroup hence $\rho(U^+)v'$ is the whole image of a polynomial function. It can be bounded if and only if it is constant. Hence $v'$ is $U^+$-invariant. 

We have just proven that every element of $V$ is $U^+$-invariant. Hence the kernel of the representation contains the normal subgroup generated by $U^+$, hence contains $N^{ss}$. 
\end{proof}

The situation is now simple: we may forget about the semisimple part because it acts trivially. And we just have an element $w$ of $V$ such that $\rho(O_n)w$ is bounded. Now for each $n$, the projection of $O_n$ in $O$ is a ball. If the $O_n$ are bounded, then $N_O=N^{ss}$ is semisimple and we are done. If not we may as in case 1 assume that the projections of $O_n$ on $O$ are increasing balls and $\rho(O_n)w$ may be bounded only if $w$ is $O$-invariant. Here $N_O$ is the subgroup generated by $N^{ss}$ and $O$ and $w$ is $N_O$-invariant.

In both cases we found the contradiction: $w$ is $N_O$-invariant. Hence the lemma \ref{lem:representation} is proved.
\end{proof}

\section{Equidistribution of dense orbits}\label{section:equidistribution}

The aim of this section is to prove the theorem \ref{the:Horbit}. First recall the theorem:
\begin{theorem*}[\ref{the:Horbit}]
Let $(G,H,\G)$ be a triple under study, $D$ a size function and $H_t$ the induced family of balls in $H$. Assume that every dominant subgroup $H'$ of $H$ has dense orbit in $\GG$.

Then the orbit of $H$ becomes equidistributed in $\GG$ with respect to $m_\GG$: $$\d{For all }\phi\in \mathcal C_c(\GG)\d{, }\frac{1}{m_H(H_t)}\int_{H_t}\phi(\pi(h))dm_H(h)\xrightarrow{t\to +\infty} \int_\GG\phi dm_\GG\;.$$
\end{theorem*}

We use in this section the rigidity of the dynamic of unipotent flows reviewed in the previous section. The article of Shah \cite{shah} is the main source of inspiration for this section.

\subsection{Equidistribution over unipotent subgroups}

The first equidistribution result is the following one: if $a_n$ is simplified and $l$ a probability measure on $U^+$, then the projections of $\dr P((a_n)_*l,t_n)$ in $\GG$ become equidistributed with respect to the Haar measure $m_\GG$ if its support $\d{Supp}((a_n)_*l,t_n)$ does not stay close to a group with closed orbit:

\begin{proposition}\label{pro:unipotent}
Let $(G,H,\G)$ be a triple under study with a size function $D$. Let $t_n$ be a sequence of positive number going to $+\infty$ and $(a_n)$ be a simplified sequence in $A^+$, $U^+$ the contracted unipotent subgroup of $H^{ss}$ and $l$ a measure on $U^+$ compactly supported and absolutely continuous with respect to the Haar measure. Let $N$ be the smallest normal subgroup of $H$ such that the projections of $\d{Supp}((a_n)_*l,t_n)$ remain in a compact subset in $H/N$. Assume eventually that $N\G$ is dense in $G$.

Then we have the following limit in the space of probability on $\GG$:
$$\lim_{n\to\infty} \pi_*(\dr P((a_n)_*l,t_n)) = m_\GG\; ,$$
that is, for every function $\phi$ continuous with compact support on $\GG$, we have:
$$\int_{H} \phi(x\G/\G) d\dr P((a_n)_*l,t_n)(x) \xrightarrow{n\to\infty}\int_\GG \phi dm_\GG\; .$$
\end{proposition}

The proof of this proposition is the core of the theorem \ref{the:Horbit}. We will use here the theory of polynomial orbits and Ratner's theorem exposed above, together with lemma \ref{lem:representation}. The derivation of theorem \ref{the:Horbit} from this proposition won't present any major difficulty.

\begin{proof}
We first show that any weak limit of the sequence studied in the proposition is a probability invariant by some unipotent subgroup. Then we will use the theory developed and the previous lemma to show that it can only be the Haar probability measure on $\GG$.

So consider the sequence of probability measures  $\pi_*(\dr P((a_n)_*l,t_n))$ and $\mu$ a weak limit. The first step will be to prove that $\mu$ is a probability measure on $\GG$. The second one will be an invariance of $\mu$ by some unipotent subgroup, thus allowing the use of the tools reviewed. Eventually we will prove that this $\mu$ is the Haar measure on $\GG$, proving the proposition.

The group $U^+$ is a unipotent subgroup of $G$ of dimension say $m$. Hence the exponential map form its Lie algebra to it is a polynomial map. Up to adding variables, we have a polynomial parametrisation $exp_1$ from  $K_S^m$ to $U^+$. The measure $l$ is absolutely continuous with respect to $(exp_+)_*(\theta_m)$.

In the same way, look at the projections of $\d{Supp}((a_n)_*l,t_n)$ to $H^u$. They are product  of balls $(H_\nu^u)_{r_n(\nu)}$ of radius some $r_n(\nu)$.
Moreover we have a polynomial parametrisation $exp_2$ from $K_S^r$ to $H^u$ which verifies that $(exp_2)_*(\theta_r)=m_{H^u}$.  And the measure $(a_n^{-1})_*[\dr P((a_n)_*l,t_n)]$  is absolutely continuous with respect to the image under $exp_1\times exp_2$ of the Haar measure $\theta_m\otimes \theta_r$ on $K_S^m\times K_S^r$.
We even get a uniformity result on the absolute continuity :

\begin{lemma}\label{lem:abscon}
Let $C$ be a positive real number. There exist a cube $B$ in $K_S^m$, a sequence of subsets $B_n$ in $K_S^r$ and an $\varepsilon>0$ such that for all measurable subset $E$ in $\GG$ we have :\\
$\textrm{If } \frac{1}{\theta_m(B)\theta_r(B_n)}\pi'_*\left((exp_1)_*(\theta_m)\otimes (exp_2)_*(\theta_r)\right)(E)\leq \varepsilon$\: \begin{flushright}then $\pi'_*((a_n^{-1})_*[\dr P((a_n)_*l,t_n)])(E)\leq \frac{C}{2}$.\end{flushright}
\end{lemma}

\begin{proof}
We choose $B$ to be a cube in $K_S^m$ such that $\Omega$ is included in $exp_1(B)$ and $B_n$ to be the preimage of $\prod_\nu H^u_{r_n(\nu)}$ under $exp_2$.

We claim that for a set of positive measure of element $\omega$ in $\Omega$, the ball $\{u\in H^u\st D(a_n \omega u)\leq t_n\}$ contains the product of balls of radius $\frac{r_n(\nu)}{2}$. This is a direct consequence of the fact that $\Omega$ is a compact  and the hypothesis made on $D$; namely the so-called orthogonality between the semisimple part and the unipotent part.

Hence in the set $exp_1(B)\times exp_2(B_n)$, the set $a_n^{-1}\d{Supp}((a_n)_*l,t_n)$ is of positive and bounded from $0$ relative measure, i.e for some $C>0$, for all n: 
$$\frac{\left((exp_1)_*(\theta_m)\otimes (exp_2)_*(\theta_r)\right)[a_n^{-1}\d{Supp}((a_n)_*l,t_n)]}{\theta_m(B)\theta_r(B_n)}\geq C
$$

As $(a_n^{-1})_*[\dr P((a_n)_*l,t_n)]$ is the restriction of the measure $l\otimes(exp_2)_*(\theta_r)$ to its support $\d{Supp}((a_n)_*l,t_n)$ renormalized to be a probability measure, and $l$ is absolutely continuous with respect to $(exp_1)_*(\theta_m)$ the conclusion of the lemma follows.
\end{proof}

{\bf Step 1:}

The measure $\mu$ is a probability measure on $\GG$.

\begin{proof}
This result is quite classical, at least in the setting of Lie groups. We will of course use the theorem \ref{the:nondiv}. Moreover it is enough to prove it for the sequence of measures $(a_n)_*(exp_1)_*(\theta_m)\otimes (exp_2)_*(\theta_r)$ restricted to $B\times B_n$ thanks to the previous lemma.

Consider the functions $\Theta_n(t,s)=a_n exp_1(t)exp_2(s)$. They are polynomials of fixed degree. Fix some $0<\epsilon<1$.
We want to find a compact set $D$ in $\GG$ such that the images of all (but a finite number) the function $\Theta_n$ are included inside this compact except for a set of relative measure at most $\epsilon$.

We claim now that the subset $D$ given to us by theorem \ref{the:nondiv} is convenient. The strategy seems clear : apply theorem \ref{the:nondiv} and then show that the second part of the alternative is impossible for all but finitely many $n$. But a difficulty appears : the sets $B\times B_n$ on which we look at the functions $\Theta_n$ are not cube. We overcome this difficulty in a somewhat artificial way : we restrict our attention $1$-parameter subgroups $O$ of $H^u$ instead of the whole $H^u$ (exactly those subgroups which appear in lemma \ref{lem:representation}). We are then able to recombine these $1$-dimensional estimates to get the wanted result.

So consider $O$ a $1$-parameter subgroup in $H^u$ and $L$ its Lie algebra in the Lie algebra of $H^u$. $L$ is a line, and (up to the choice of a basis vector in $L$) for $n$ big enough, $L_n=L\cup B_n$ is a ball in $K_S$.

Hence we may apply theorem \ref{the:nondiv} to the functions $\Theta_n$ restricted to $B\times L_n$ which is a cube. And we know, using lemma \ref{lem:representation}, that the action of $a_n exp_1(B)\times exp_2(L_n)$ sends the points $v_k$ outside of the compact $D_k$ unless it is invariant by the group $N_O$.

So for $n$ big enough, either all the points $v_k$ appearing in theorem \ref{the:nondiv} are invariant under $N^{ss}$ and $O$ or  the whole cube $B\times L_n$ but a set of relative measure at most $\epsilon$ is mapped inside $D$.
Now the first part of the alternative means that the subgroup $N_O$ is included in the intersection of the parabolic subgroups $P_k$ and as a corollary its orbit in $\GG$ is closed. So this may happen only along a negligeable set of directions $O$ : as any set $B$ of positive measure of directions $O$ generates $H^u$ itself, the $N_O$'s for $O$ in $B$ generates the subgroup $N$ (smallest normal subgroup such that $\d{Supp}((a_n)_*l,t_n)$ is bounded in $H/N$). And we assumed the group $N$ has a dense orbit in $\GG$.

So for $n$ big enough, the total mass of points $(t,s)\in B\times L_n$ such that $\Theta_n(t,s)$ does not belong to $D$ does not exceed $2 \epsilon$ times the mass of $B\times L_n$.
\end{proof}

{\bf Step 2:}

The probability measure $\mu$ is left-invariant by some unipotent subgroup $Z$.

\begin{proof}
We also handle differently the cases according to the behaviour of $a_n$:

\emph{Case 1 : $a_n$ is bounded} 

We may assume that $a_n$ is constantly equal to $Id$ and $U^+$ is restricted to $\{Id\}$. Hence the set $U_n=\d{Supp}(Id,t_n)$ is an increasing sequence of balls of radius $t_n$ in $H^u$ and the probability measure $\dr P(\omega, t_n)$ is the Haar measure of $H^u$ restricted to $U_n$. 

Let $Z$ be the center of the unipotent group $H^u$, $\mathfrak z$ its Lie algebra. Then the polynomial map $P$ given by the composition of the representation choosen to define the norm and the exponential map from the Lie algebra $\mathfrak h^u$ to $H^u$ is proper and verifies for all $z\in \mathfrak z$ and $u\in \mathfrak h^u$:
$$P(z+u)=P(z)P(u)$$

Hence, for a fixed $z\in \mathfrak z$, the "norm" $D(\exp(z)\exp(u))$ is equivalent to  $D(\exp(u))$, as $P(z+u)=P(u)+O(P'(u))$. This proves that the ratio $\frac{m_{H^u}(\d{exp}(z)U_n\cap U_n)}{m_{H^u}(U_n)}$ tends to $1$, which means that $\mu$ is left-invariant by $Z$. 

\smallskip

\emph{Case 2 : $a_n$ is not bounded}

Then, by construction $a_n$ has a contracting action on $U^+$. Moreover $u^+_*\mu$ is the limit of $u^+_*\pi_*(\dr P((a_n)_*l,t_n))$. And the last one may be rewritten $\pi_*\left(\dr P[(a_n)_*(a_n^{-1}u^+a_n)_*l,t_n]\right)$. As $a_n^{-1}u^+a_n$ goes to $Id$,  lemma \ref{lem:limsupport} implies that $\mu$ is $U^+$ invariant.
\end{proof}

Hence we may use all the tools presented: there exists a class $\mathcal F$-subgroup ${\bf P}$ of $\bG$ such that $\mu(X(P,Z))$ is positive. We want to show that ${\bf P}=\bG$. This is the third and final step:

{\bf Step 3:}

Any class $\mathcal F$-subgroup ${\bf P}$ such that $\mu(X(P,Z))>0$ is the group $G$.

\begin{proof}
We will naturally use the theorem \ref{the:DM}. Fix a compact $C$ of $X(P,V)$ of positive measure.

Using lemma \ref{lem:abscon}, we get an $\varepsilon$ such that for all measurable subset $E$ in $C$ we have :

$\textrm{If } \frac{1}{\theta_m(B)\theta_r(B_n)}\pi_*\left((exp_1)_*(\theta_m)\otimes (exp_2)_*(\theta_r)\right)(E)\leq \varepsilon$
\begin{flushright} $\textrm{ then }\pi_*\left((a_n^{-1})_*[\dr P((a_n)_*l,t_n)]\right)(E)\leq \frac{\mu(C)}{2}\;$\end{flushright}

Once again we will apply the theorem \ref{the:DM} directionally to the function $\Theta_n(t,s)=a_n exp_1(t)exp_2(s)$, restricted to some $B\times L_n$ ($L_n$ being a ball in the Lie algebra $L$ of a $1$-parameter subgroup $O$ in $H^u$), to the compact $C$ and the $\varepsilon$ just defined. Note $\theta$ a normalized Haar measure on $L$. Then there exists a compact $D$ of $F(P,V)$ such that for all neighborhood $W_0$ of D there exists a neighborhood $W$ of $C$ such that for all $n$ we get the alternative:
\begin{itemize}
\item There exists $\gamma_n$ in $\Gamma$ such that $\eta(\Theta_n(B\times L_n) \gamma_n)\subset W_0$
\item $\theta_m\otimes \theta(\{t,s\in B\times L_n \d{ such that } \Theta_n(t,s)\G/\G \in W\})<\varepsilon \theta_m\otimes \theta(B\times L_n)$
\end{itemize} 

Now fix any neighborhood $W_0$ of $D$ and suppose we are in the second case of the previous alternative. Then by construction, we have: 
$$
\frac{1}{\theta_m\otimes \theta(B\times B_n)}\pi_*((exp_1)_*(\theta_m)\otimes (exp_2)_*(\theta))(a_n^{-1})(W)<\varepsilon\; .
$$

But we have $\pi_*(\dr P((a_n)_*l,t_n))(W)>\frac{\mu(C)}{2}$ as $W$ contains $C$ and the measures $\pi_*(\dr P((a_n)_*l,t_n))$ converges to $\mu$. By definition of $\varepsilon$, there is a set of $1$-parameter subgroups $O$ of positive measure for which the previous inequality does not hold for all $n$ big enough.

Now we use the lemma \ref{lem:representation}. Consider the representation $\rho$ of $G$ in the $K$ module $V$ associated to ${\bf P}$ via Chevalley's theorem. And restrict it to a representation of $H$ in $V$. Let $\Lambda$ be the discrete set $\eta(h_0\Gamma)$. We want to show that one of this point is invariant under the action of the group $N_O$ (see lemma \ref{lem:representation}). But this is a direct application of the lemma \ref{lem:representation}: the sets $\rho(a_nB\times L_n)\eta(\gamma_n)=\eta(\Theta_n(B\times B_n))$ are included in $W_0$ hence bounded. The conclusion of lemma \ref{lem:representation} being violated, the hypothesis is not fulfilled: one of the points $\gamma_n$ is $N_O$-invariant.

And now we may intervert the quantifiers without loosing everything : there is an integer $n$ such that $\gamma_n$ is invariant under $N_O$ for a set of positive measure of directions $O$. As previously noted, this point is $N$-invariant ($N$ being the smallest normal subgroup of $H$ containing all the sets $\d{Supp}((a_n)_*l,t_n)$ in a compact neighbourhood).

We have done most of the work. Let us conclude, using notations and results of the section \ref{sec:ratner}: $N$ is included in $\gamma_n^{-1} N_1(P) \gamma_n$. So the projection of $N_1(P)$ in $\GG$ contains a translate of the projection of $N$. But the latter is dense and the first one is closed : $N_1(P)$ projects onto $\GG$ hence is Zariski-dense in $G$. We conclude that $N_1(P)=G$. That means that ${\bf P}$ is a normal subgroup of $\bG$, so is equal to $\bG$ by simplicity.
\end{proof}

To conclude the proof of the proposition \ref{pro:unipotent}, note that the rigidity theorem \ref{the:ratner} implies that $\mu$ is invariant under some finite index subgroup $P$ of $G$. As $G$ is a simply connected group, $G$ itself is the unique finite index subgroup of $G$. Eventually $\mu$ is $G$-invariant so is the Haar probability measure on $\GG$.
\end{proof}

\subsection{Equidistribution of spheres}

We need a last step before proving theorem \ref{the:Horbit} : that is a proposition very similar to proposition \ref{pro:unipotent} but more adapted to Cartan decomposition in the group $H^{ss}$. Recall that, at the begining of section \ref{section:equidistribution}, we defined the Cartan decomposition $H^{ss}=C D A^+ C$. The following proposition holds (compare with \cite[Corollary 1.2]{shah}):

\begin{proposition}\label{pro:spheres}
Let $(G,H,\G)$ be a triple under study. Let $(h_n)$ be a sequence in $H^{ss}$, $t_n$ a sequence of positive number going to $+\infty$ and $\mu$ a probability measure on $C$ absolutely continuous with respect to the Haar probability measure on $C$. We assume that for all $c$ in the support of $\mu$, we have $D(h_n c)\leq (1-\varepsilon)t_n$ for some $\varepsilon>0$. Let $N$ be the smallest normal subgroup of $H$ such that the projection of the support of $\dr P((a_n)_*\mu, t_n)$ is bounded in $H/N$. Assume that $\G N$ is dense in $G$. 

Then the projection of probability measures $\dr P((a_n)_*\mu, t_n)$ in $\GG$ becomes equidistributed:
$$\lim_{n\to\infty} \pi_*(\dr P((a_n)_*\mu, t_n)) = m_\GG\; ,$$
that is, for every function $\phi$ continuous with compact support on $\GG$, we have:
$$\int_H \phi(h \G/\G) d\dr P((a_n)_*\mu, t_n) \xrightarrow{n\to\infty}\int_\GG \phi dm_\GG\; .$$
\end{proposition}

\begin{proof}
We will prove that any weak limit of this sequence of probability measure is the Haar measure $m_\GG$.

First of all, we may assume that $h_n$ is an element of $A^+$. Indeed, using Cartan decomposition, we write $h_n=c^1_n d_n a_n c^2_n$, and, up to an extraction, the three sequences $c^1_n$, $c^2_n$ and $d_n$ converge to respectively $c^1$, $c^2$ and $d$. Now, let $\mu'$ be the pushforward of $\mu$ under $c^2$ : $\mu'(c^2A)=\mu(A)$. The lemma \ref{lem:limsupport} guarantees that the equidistribution of $\pi_*(\dr P((h_n)_*\mu, t_n))$ is equivalent to the one of $\pi_*(\dr P((a_n)_*\mu', t_n))$. And by construction, $N$ is also the smallest normal subgroup such that the projection of $\d{Supp}((a_n)_*\mu', t_n)$ is bounded in $H/N$.

Moreover, up to another extraction, we assume that $a_n$ is simplified.
Consider now the opposite parabolic subgroup $P^-$ to $U^+$ in $H^{ss}$ and $U^-$ the expanded unipotent subgroup : 
$$U^- = \left\{ h\in H^{ss} \st \lim_{n\to +\infty}a_n h a_n^{-1} =e \right\}\, .$$
Every neighbourhood of an element $c$ in $C$, contains a neighbourhood which is homeomorphic to a neighbourhood of $Id$ in $P^-\times U^+$ via the application $(p^-,u^+)\mapsto p^-u^+c$. We may split the support of $\mu$ in such sets (up to a negligible set), or in other words, we assume $\mu$ to be supported inside an open set homeomorphic to an open set $\Omega^-\times \Omega^+$ in $P^-\times U^+$. We furthermore assume that both $\Omega^-$ and $\Omega^+$ are product set of the form $\displaystyle \prod_{\nu\in S}\Omega_\nu$. Moreover at the archimedean places, we may "thicken" a little bit $\mu$ to construct a measure absolutely continuous with respect to $m_{H^{ss}}$ : let $\lambda$ be a probability measure on a sufficently small neighbourhood $O$ of $Id$ in $U^-_\infty$ (the archimedean part of $U^-$) absolutely continuous with respect to the Haar measure on $U^-_\infty$. Then $\lambda \otimes \mu$ is absolutely continuous with respect to the Haar measure on $H^{ss}$ (see \cite[Page 15]{shah}).

Looking at the action of $a_n$ on $U^-$ and using lemma \ref{lem:limtranslate} it is clear that for every function $f$ continuous with compact support in $\GG$, the integrals of $f$ for the both measures $\pi_*(\dr P((a_n)_*\lambda \otimes\mu,t_n))$ and $\pi_*(\dr P((a_n)_*\mu,t_n))$ are equivalent as $n$ go to $\infty$:\\
$\,$\\
$
|\int_\GG f d\pi_*(\dr P((a_n)_*\lambda \otimes\mu,t_n))- \int_{U^+\times H^u} f(x) d\pi_*(\dr P((a_n)_*\mu,t_n))|$\\

\smallskip

$
\leq  \int_{U^-}|\int_H f(x) d\pi_*\left((\dr P((a_n o a_n^{-1})a_n)_*\mu,t_n)-(\dr P((a_n o a_n^{-1})a_n)_*\mu,t_n)\right)(x)|d\lambda(o)$\\
\begin{eqnarray}\label{eqn:compare}
\xrightarrow{n\to\infty} 0
\end{eqnarray}

The limit is obtained using $a_n o a_n^{-1}\xrightarrow{n\to \infty} Id$, lemma \ref{lem:limtranslate} and the dominated convergence theorem.

 We work now with $\lambda\otimes \mu$. Remark that, at non-archimedean places, we do not have to modify $\mu$, as maximal compact subgroups are also open.

Now, using \cite[Proposition 6.1]{shah}, we may decompose this probability measure $\lambda \otimes\mu$ in the product $\Omega^-\times \Omega^+$ : there are a probability measure $\nu^-$ on $\Omega^-$ and for almost all $x$ in $\Omega^-$, a probability measure $\nu^+_x$ on $\Omega^+$ such that :
\begin{itemize}
\item $\nu^-$ and all the $\nu^+_\omega$ are absolutely continuous with respect to the Haar measure on $P^-$ and $U^+$ respectively.
\item for all $\phi$ continuous with compact support in $H^{ss}$, we have 
$$\int_{H^{ss}} \phi d(\lambda\otimes \mu)=\int_{\Omega^-}\int_{\Omega^+} \phi(xy) d\nu^+_x(y)d \nu^-(x)\, .$$
\end{itemize}

Consider now a function $f$ continuous with compact support in $\GG$. We have:
$$
\int_\GG f d\pi_*(\dr P((a_n)_*\lambda \otimes\mu,t_n))  =  \int_{\Omega^-}\int_{\Omega^+\times H^u} f(y\G/\G) d\dr P((a_nx)_*\nu^+_x, t_n)(y)d\nu^-(x)
$$

So the last difficulty that remains is to compare the two probability measures $\dr P((a_nx)_*\nu^+_x, t_n)$ and $\dr P((a_n)_*\nu^+_x, t_n)$ : if we prove that they are sufficiently close, then we may use the proposition \ref{pro:unipotent} to conclude that the limit is the Haar probability measure $m_\GG$. But under conjugacy by $a_n$, the elements in $P^-$ remains bounded. So, if we choose the support of $\lambda$ small enough, the lemma \ref{lem:limtranslate} ensures that the two measures $\dr P((a_nx)_*\nu^+_x, t_n)=\dr P((a_n x a_n^{-1})_*(a_n)_*\nu^+_x, t_n)$ and $\dr P((a_n)_*\nu^+_x, t_n)$ are arbitrarily closed. 

Fix $\epsilon>0$ and choose the support $O$ of $\lambda$ such that we have : for all $x\in O$, all $n$ $$\left|\int_{\Omega^+\times H^u} f(y\G/\G) d\dr P((a_nx)_*\nu^+_x, t_n)(y)-\int_{\Omega^+\times H^u} f(y\G/\G) d\dr P((a_n)_*\nu^+_x, t_n)(y)\right|\leq \epsilon$$

Then, we have :
$$
\left| \int_\GG f d\pi_*(\dr P((a_n)_*\lambda \otimes\mu,t_n)) - \int_{\Omega^-}\int_{\Omega^+\times H^u} f(y\G/\G) d\dr P((a_n)_*\nu^+_x, t_n)(y)d\nu^-(x)\right|\leq \epsilon
$$

Now, the proposition \ref{pro:unipotent} states that for all $x$, we have the limit:
$$\int_{\Omega^+\times H^u} f(y\G/\G) d\dr P((a_n)_*\nu^+_x, t_n)(y)\xrightarrow{n\to\infty}\int_\GG f dm_\GG$$
We conclude applying the dominated convergence theorem:
$$
\left|\int_\GG f d\pi_*(\dr P((a_n)_*\lambda \otimes\mu,t_n)) - \int_{\GG}fdm_\GG\right|\leq \epsilon
$$

So the previous inequality together with \ref{eqn:compare} leads to (for $n$ big enough):
$$
|\int_\GG f d\pi_*(\dr P((a_n)_*\mu,t_n)) - \int_{\GG}fdm_\GG|\leq 2\epsilon
$$
As this is true for arbitrary $\epsilon$, we have finally obtained the desired result :
$$
\int_\GG f d\pi_*(\dr P((a_n)_*\mu,t_n)) \xrightarrow{n\to \infty} \int_{\GG}fdm_\GG\, .
$$

The proposition is proven
\end{proof}

Thanks to this proposition, we are able to define a subset of large relative volume in $H$ such that, basically, as soon as the support of $\dr P(h_*m_C,t)$ hits this subset, the projection of this measure in $\GG$ is closed to the Haar probability measure (once again, the statement is a bit more complicated than what I just explained but this will be the exact result needed):

\begin{cor}\label{coro:spheres}
Let $(G,H,\G)$ be a triple under study, together with a size function $D$. Assume that every dominant normal subgroup of $H$ has a dense orbit in $\GG$.
Fix $\varepsilon>0$, $f$ a continuous function with compact support in $\GG$, and $O$ some open subset in $C$.Then there is a finite number of non-dominant normal subgroups $N_1$, $\ldots$, $N_k$ of $H$, a compact subset $B$ in $H$ such that:

For $h$ in $H$, $O' \subset C$ containing $O$ with $\mu$ the probability measure on $O'$ proportional to the Haar measure on $C$ and $t>0$ verifying for all $o$ in $O$, $D(go)\leq \frac{t}{1+\varepsilon}$, we have:

If the support of $\dr P(h_*\mu,t)$ is not included in any $B N_i$, then 
$$|\int_\GG f d \pi_*(\dr P(h_*\mu,t))-\int_\GG f dm_\GG|\leq \varepsilon$$
\end{cor}

\begin{proof}
Take the $N_i$'s to be the maximal normal non dominant subgroups. They are in finite number. Suppose they do not verify the corollary. Then we construct a sequence $h_n$, $t_n$, $O_n$ such that the supports of $\dr P((h_n)_*\mu_n,t_n)$ are not included in any compact neighbourhood of a non-dominant normal subgroup and the difference of integrals is always greater than $\varepsilon$:
\begin{eqnarray}\label{eqn:diff}
|\int_\GG f d \pi_*(\dr P((h_n)_*\mu_n,t_n))-\int_\GG f dm_\GG|> \varepsilon
\end{eqnarray} 

Up to an extraction, we may assume that $\mu_n$ converges to a measure which is equal the probability measure $\mu_\infty$ on an open $O'_\infty$ containing $O$ and proportional to the Haar measure of $C$.

These supports are yet included in a compact neighbourhood of some normal subgroup $N$ which has to be dominant. By assumption, $N$ has a dense projection in $\GG$. So we may apply the above proposition to this sequence: the projection $\pi_*(\dr P((h_n)_*\mu_\infty,t_n))$ converges to the Haar probability measure $m_\GG$. But, using lemma \ref{lem:limsupport}, letting $n$ go to infinity, the measure $\pi_*(\dr P((h_n)_*\mu_\infty,t_n))$ is arbitrarily closed to $\pi_*(\dr P((h_n)_*\mu_n,t_n))$ This contradicts \ref{eqn:diff}.
\end{proof}

\subsection{Equidistribution of balls}

At last we are able to conclude the proof of equidistribution of balls. Fix a function $f$ continuous with compact support in $\GG$. Fix $\varepsilon>0$. Let $\eta>0$ be such that $\frac{m_H(H_{(1+\eta) t})}{m_H(H_t)}\leq 1+\varepsilon$ for all $t$.

There is a neighborhood $O$ of $Id$ in $C$ such that for all $h\in H^{ss}$ we have $D(ho)\leq \sqrt{1+\eta}D(h)$. And we may choose $O$ such that $C$ is a disjoint union of translates of $O$ (up to a negligible set): there exist $c_1,\ldots, c_s$ such that $c_iO\cap c_jO$ has measure $0$ and the union $\cup_1^s c_i O$ is of full measure in $C$. Note $\mu_O$ the restriction of the probability Haar measure of $C$ to $O$.

Let $\tilde H_t$ be the union over $c\in CD$, $a\in A^+$, and $1\leq i\leq s$ with $D(cac_i)\leq t$, of the support of $m((cac_i)_*\mu_{O},(1+\eta)t)$. Thanks to the Cartan decomposition, up to a negligible set, $\tilde H_t$ contains $H_t$, is contained in $H_{(1+\eta) t}$ and the restriction of $m_H$ to $\tilde H_t$ may be written:
$$ (m_H)_{|\tilde H_t} =\sum_1^s \int_{c\in CD ,\; a\in A,\; D(cac_i)\leq t} m((cac_i)_*\mu_O,(1+\eta)t)$$
 
Let $E_t$ be the union of the supports of measures $m((ca)_*\mu_{ca},(1+\eta)t)$ which are completely included in $\displaystyle B\bigcup_1^k N_i$ (the sets constructed in the above corollary). As none of the $N_i$'s are dominant, for $t$ big enough, the relative mass of $E_t$ in $\tilde H_t$ is less than $\varepsilon$ and the symmetric difference between $H_t$ and $\tilde H_t\setminus E_t$ is almost negligible:
$$\frac{m_H(H_t \Delta (\tilde H_t \setminus E_t))}{m_H(H_t)}\leq 2\varepsilon $$

Corollary \ref{coro:spheres} implies that for all $a\in A$, $c\in CD$ and $1\leq i\leq s$, if the support $\d{Supp}(m((cac_i)_*\mu_O, (1+\eta)t))$ is not included in $E_t$, then its projection is pretty well distributed: 
$$\left|\int_H f\; dm((cac_i)_*\mu_O, (1+\eta)t)- m((cac_i)_*\mu_O, (1+\eta)t)(H)\int_\GG f \;dm_\GG\right|\leq \varepsilon$$

Integrating all these approximation over $c$, $a$ and $c_i$ leads to:
$$|\frac{1}{m_H(\tilde H_t\setminus E_t)}\int_{\tilde H_t\setminus E_t} fd\pi_*(m_H) -\int_\GG fdm_\GG|\leq \varepsilon \;.$$

So going back to the desired integral, we get (for $t$ big enough):
$$|\frac{1}{m_H(H_t)}\int_{H_t} fd\pi_*(m_H) -\int_\GG fdm_\GG|\leq (1+4\max(|f|))\varepsilon \;.$$
As $\epsilon$ is arbitrarily small, we get the desired result:
$$\frac{1}{m_H(H_t)}\int_{H_t} fd\pi_*(m_H) \xrightarrow{t\to \infty}\int_\GG fdm_\GG\;.$$

This concludes the proof of theorem \ref{the:Horbit}.
 
\section{Applications}\label{sec:exa}

We conclude this text by some explanations on the applications described in the introduction.

\subsection{In dimension 2}

Recall the framework: we consider the group $G=\d{SL}(2,\dr R)\times \d{SL}(2,\dr Q_p)$ for $p$ a prime number, and the lattice $\G=\d{SL}(2,\dr Z[\frac{1}{p}])$. We fix here (for sake of simplicity) the standard euclidean norm $| . |_\infty$ on the matrix algebra $\mathcal M(2,\dr R)$ and the max-norm $|.|_p$ on $\mathcal M(2,\dr Q_p)$. For a point $v$ in $\dr R^2$, we note also $|v|_\infty$ the norm of the matrix whose first column is $v$ and the second one is $0$. We define similarly the norm of a point in $\dr Q_p^2$. We choose a Haar measure $m=m_\infty\otimes m_p$ on $G$.

The first result was:

\begin{appl*}[\ref{appl21}]
Let $O$ be a bounded open subset of $\d{SL}(2,\dr Q_p)$. Note $\G^O_T$ the set of elements $\g\in \G$ such that $|\g|_\infty \leq T$ and $\g\in O$ as an element of $\d{SL}(2,\dr Q_p)$. Let $v$ be a point of the plane  $\dr R^2\setminus \{0\}$ with coordinates independant over $\dr Q$.

Then we have the following limit, for any function $\phi$ continuous with compact support in $\dr R^2\setminus \{0\}$:
$$ \frac{1}{T} \sum_{\G_T^O} \phi(\g(v)) \xrightarrow{T\to \infty} \frac{m_p(O)}{m(\GG)|v|_\infty}\int_{\dr R^2} \phi(w)\frac{dw}{|w|_\infty}$$
\end{appl*}

\begin{proof}
We work here in the product of $\dr R^2\setminus\{0\}$ and $\d{SL}(2,\dr Q_p)$. We see it as  the homogeneous space $\GH$ with $H=\d{Stab}(v)$  the stabilizator of $v$ for the linear action of $\d{SL}(2,\dr R)$ on the plane.

Then it is not difficult to see that the hypothesis on the norm are fulfilled and that $H$ has no dominant subgroup except itself. Moreover the volume of balls are explicitely computed: the ratios of $m_H(H_tg)$ and $m_H(H_t)$ tends to $\frac{1}{|v||w|}$ where $w=g(v)$ \cite[Section 12.4]{goroweiss}. Remark that there is no need here to split the parameter space.

It remains to prove that $H.\d{SL}(2,\dr Z[\frac{1}{p}])$ is dense in $G$. But it contains $Stab(v).\d{SL}(2,\dr Z)$ which is by hypothesis dense in $\d{SL}(2,\dr R)$. Now we may use the strong aproximation in $\d{SL}(2)$ \cite{platonov-rapinchuk}: the algebraic group $\d{SL}(2)$ is semisimple simply connected, hence the product $\d{SL}(2,\dr R).\d{SL}(2,\dr Z[\frac{1}{p}])$ is dense in $G$. This yields the desired property: $H.\d{SL}(2,\dr Z[\frac{1}{p}])$ is dense in $G$.

Now, theorem \ref{the:duality} implies the stated result.
\end{proof}

The second application was the following one. Recall that on the $p$-adic plane, we normalize the measure such that it gives mass $1$ to $\dr Z_p^2$. The result is that if your beginning point generates the whole plane among the $\dr Q$-subspaces, then its orbit is dense and you get a distribution result (the function $E$ appearing is the integer part):

\begin{appl*}[\ref{appl22}]
Let $(v_\infty,v_p)$ be an element of $(\dr R^2\setminus{0})\times(\dr Q_p^2\setminus{0})$. Suppose that any $\dr Q$-subspace $V$ of $\dr Q^2$ verifying $v_\infty\in V\otimes_{\dr Q} \dr R$ and $v_p \in V\otimes_{\dr Q} \dr Q_p$ is $\dr Q^2$. Denote $\G_T$ the set of elements $\g\in \G$ with $|\g|_\infty\leq T$ and $|\g|_p\leq T$.

Then, for all function $\phi$ continuous with compact support in $(\dr R^2\setminus{0})\times(\dr Q_p^2\setminus{0})$, we have the following limit:\\
$\frac{1}{T p^{E(\ln_p(T))}} \sum_{\G_T} \phi(\g v_\infty, \g v_p)\xrightarrow{T\to\infty}$
\begin{flushright} $\frac{p^2-1}{p^2 m(\GG)|v_\infty|_\infty |v_p|_p} \int_{\dr R^2\times \dr Q_p^2} \phi(v,w) \frac{dv dw}{|w|_\infty |w|_p}$\end{flushright}
\end{appl*}

\begin{proof}
The proof here is similar to the previous one, the group $H$ being $\d{Stab}(v_\infty)\times \d{Stab}(v_p)$. The hypothesis on the norm are fulfilled, as $H$ is unipotent. The volume ratio limits are easy to compute and left to the reader. You just have to be careful with the normalizations of measures, letting appear this constant $\frac{p^2-1}{p^2}$.

So it just remains to prove that $H\d{SL}(2,\dr Z[\frac{1}{p}])$ is dense in $G$. The key point is that its closure	 must be (up to finite index) the $\dr R\times \dr Q_p$-points of a $\dr Q$-subgroup of $\d{SL}(2)$, by Tomanov theorem : it is a closed subset in $\GG$ invariant under unipotent subgroups.

Hence, if either $v_\infty$ or $v_p$ has coordinates independant over $\dr Q$, the argument in previous application show the density. The only remaining case is when both $v_\infty$ and $v_p$ are stabilized by a $\dr Q$-unipotent group. But the assumption that $v_\infty$ and $v_p$ "generates" $\dr Q^2$ is then equivalent to the fact that these two stabilizers are different. Now we may conclude, arguing that two different unipotent subgroups of $\d{SL}(2,\dr Q)$ generate the whole group. Hence the smallest $\dr Q$-subgroup of  $\d{SL}(2,\dr Q)$ such that its real points contains the stabilizer of $v_\infty$ and its $p$-adic the stabilizer of $v_p$ is $\d{SL}(2)$. And the closure of $H.\d{SL}(2,\dr Z[\frac{1}{p}])$ is $G$.
\end{proof}

The two previous examples showed how to profit of both the rigidity of orbit closures in an $S$-arithmetic setting and algebraic featurees such as strong approximation in the ambient group $G$. These arguments are also the core of the next case.

\subsection{In greater dimension}

Recall that we look at the action of $\G=\d{SL}(n,\dr Z)$ on the $k$-th exterior power $\Lambda^k(\dr R^n)$. And we fix the standard euclidean norm $|.|$ on $\mathcal M(n,\dr R)$. We consider also the standard euclidean norm on $\Lambda^k(\dr R^n)$ and $m$ is a Haar measure on $\d{SL}(n,\dr R)$.
We want to prove:

\begin{appl*}[\ref{appln}]
Let $v$ be a non-zero element of $\Lambda^k(\dr R^n)$ such that its corresponding $k$-plane of $\dr R^n$ contains no rational vector. Denote $\G_T$ the set of elements $\g\in \G$ with $|\g|\leq T$.

Then we have a positive real constant $c$ (independant of $\G$ and $v$) such that for all function $\phi$ continuous with compact support on $\Lambda^k(\dr R^n)\setminus \{0\}$:
$$\frac{1}{T^{n^2+k^2-nk-n}}\sum_{\G_T} \phi(\g v) \xrightarrow{T\to\infty} \frac{c}{m(\GG)|v|}\int_{\Lambda^k(\dr R^n)} \phi(v') \frac{dv'}{|v'|}$$
\end{appl*}

\begin{proof}
Here we have to be more careful than in previous section. We consider the subgroup $H = \d{Stab} (v)$. It is a conjugate of the group $H_0$ of the form:
$$H_0=\left( \begin{matrix} \d{SL}(k,\dr R) & H^u\\ 0 & \d{SL}(n-k,\dr R)\end{matrix}\right):=\left( \begin{matrix} H_k & H^u\\ 0 & H_{n-k}\end{matrix}\right)$$

So it is a semidirect product of a semisimple and a unipotent group. Moreover the quotient $H_0\backslash G$ identifies with $\Lambda^k(\dr R^n)\setminus \{0\}$ via the projection associating at an element of $\d{SL}(n,\dr R)$ the exterior product of its $k$-first lines.

We have to prove the orthogonality property for the norm on $H=gH_0g^{-1}$ ($g\in \d{SL}(n,\dr R)$). The key point is that one may use Iwasawa decomposition to write $g=o a n$ where $o$ belongs to $\d{SO}(n)$, $a$ is diagonal and $n$ is upper triangular and nilpotent so an element of $H_0$. By bi-invariance of the euclidean norm under $\d{SO}(n)$, and the fact that $a$ normalizes the semisimple part of $H$ and the unipotent one, we get, for $h=gh_0g^{-1}$ with the obvious notation:
$$|h|^2= |a h_0 a^{-1}|^2 =|a h_0^{ss} a^{-1}|^2+|a h_0^{u} a^{-1}|^2=|h^{ss}|^2+|h^u|^2$$.

Now it is clear that $H_0$ has no dominant subgroup except itself, so the same holds for $H$. Let us prove that $H\G$ is dense before evaluating the volume ratio limits. The simplest way to see it is to pull back this dynamic on the space of $k$-frames : choose a family of $k$ vectors in $\dr R^n$ such that their exterior product is $v$. Then the hypothesis on $v$ is that the $k$-plane generated by this family of vector contains no non-zero rational vectors. By a theorem of Dani and Raghavan \cite{dani-raghavan}, it implies that the orbit of this family under $\G$ is dense in the space of $k$-frames. This in turn implies by projection that the orbit of $v$ under $\G$ is dense in $H_0\backslash G$, i.e. that $H\G$ is dense in $G$.

\smallskip

We have compute the volume ratios to get the limiting density. Precisely, let $w=H_0g' =g^{-1}H(g'g^{-1})$ be a non-zero point in $\Lambda^k(\dr R^n)$. Then the limiting density at $w$ given by theorem \ref{the:duality} is the ratio: $$\frac{m_H(H_t(g'g^{-1})}{m_H(H_t)}$$

The set $H_t(g'g^{-1})$ is by definition $\{h\in H \st |hgg'|\leq t\}$ ; or the set $\{h_0\in H_0 \st |gh_0g'|\leq t\}$. Hence we have to compute the measure $M_t(g,g')=m_{H_0}(\{h_0\in H_0 \st |gh_0g'|\leq t\})$. The choice of normalization of $m_{H_0}$ has no importance, as we only want to compute ratios.
Using the bi-invariance of the norm and the Iwasawa decomposition of $g$ and $g'^{-1}$, we immediatly see that $M_t(g,g)=\frac{1}{|\d{Vol}(g)| |\d{Vol}(g')|}M_t(1,1)$, where $\d{Vol}(g)$ is the determinant of the $k$ first line of $g$. And, by the definition of the exterior product, the absolute value of this determinant is the euclidean norm of their exterior product. So we may rewrite $M_t(g,g')=\frac{1}{|v||w|}M_t(1,1)$. This gives the limiting density.

At this point, we need a last estimation: an equivalent of $M_T(1,1)$ which gives the renormalisation factor $T^{n^2+k^2-nk-n}$. So we want to compute the volume of the set $\{h_0\in H_0 \st |h_0|\leq T\}$ for the standard Haar measure on $H_0$: the product of the standard Haar measure on the three groups $\d{SL}(k,\dr R)$, $\d{SL}(n-k,\dr R)$ and $H^u$. Using the estimations of Maucourant \cite{maucourant}, we see that the volume of the sphere of radius $T$ in these groups are respectively of order $T^{k^2-k-1}$, $T^{(n-k)^2-(n-k)-1}$ and $T^{k(n-k)-1}$.
So the leading term of the volume of the ball of radius $T$ is of order:
$$\int_{T_1^2+T_2^2+T_3^2\leq T^2} T_1^{k^2-k-1}T_2^{(n-k)^2-(n-k)-1}T_3^{k(n-k)-1}$$
Hence the leading term is of order: $$T^{k^2-k+(n-k)^2-(n-k)+n(n-k)}=T^{n^2+k^2-nk-n}$$

This concludes the proof of application \ref{appln}
\end{proof}

I conclude this article with the $S$-arithmetic generalization of the previous result. I leave the proof to the reader. All the arguments are in the three previous proofs except an estimation of the volume of the ball of radius $T$ in $\d{SL}(k,\dr Q_p)$ ($p$ being a prime number). Using Cartan decomposition and some basic combinatorics, we get that the leading term of this volume is $(p^{E(\ln_p(T)})^{k^2-k}$. We fix the max norm in the standard basis on $\mathcal M(n,\dr Q_p)$ and $\Lambda^k(\dr Q_p^n)$. The group $\G$ is $\d{SL}(n,\dr Z[\frac{1}{p}])$, and we note for an element $\g\in \G$, $|\g|$ the max of its real euclidean norm and $p$-adic max norm.

\begin{appl}
Let $v=(v_\infty,v_p)$ be a non-zero element of $\Lambda^k(\dr R^n\times \dr Q_p^n)$ such that there is no non-zero rational vector belonging to both the real $k$-planes associated to $v_\infty$ and the $p$-adic one associated to $v_p$. Denote $\G_T$ the set of elements $\g\in \G$ with $|\g|\leq T$.

Then we have a positive real constant $c$ (independant of $\G$ and $v$) such that for all function $\phi$ continuous with compact support on $\Lambda^k(\dr R^n)\setminus \{0\}$:

$\frac{1}{(Tp^{E(\ln_p(T))})^{n^2+k^2-nk-n}}\sum_{\G_T} \phi(\g v)\xrightarrow{T\to\infty}$
\begin{flushright}$\frac{c}{m(\GG)|v_\infty|_\infty}\int_{\Lambda^k(\dr R^n\times \dr Q_p^n)} \phi(v_\infty',v'_p) \frac{dv_\infty'v_p'}{|v_\infty'|_\infty|v_p|_p}$\end{flushright}
\end{appl}

\bibliographystyle{plain}
\bibliography{biblio}

\begin{thebibliography}{10}

\bibitem{benoist-oh}
Y.~Benoist and H.~Oh.
\newblock Effective equidistribution of {S}-integral points on symmetric
  varieties.
\newblock arxiv:0706.1621.

\bibitem{Borel}
A.~Borel.
\newblock {\em Linear Algebraic Groups}.
\newblock Mathematics Lecture Note Series, New York, 1969.

\bibitem{dani-raghavan}
S.~G. Dani and S.~Raghavan.
\newblock Orbits of {E}uclidean frames under discrete linear groups.
\newblock {\em Israel J. Math.}, 36(3-4):300--320, 1980.

\bibitem{denef}
J.~Denef.
\newblock On the evaluation of certain {$p$}-adic integrals.
\newblock In {\em S\'eminaire de th\'eorie des nombres, {P}aris 1983--84},
  volume~59 of {\em Progr. Math.}, pages 25--47. Birkh\"auser Boston, Boston,
  MA, 1985.

\bibitem{venkatesh-ellenberg}
J.S. Ellenberg and A.~Venkatesh.
\newblock {Local-global principles for representations of quadratic forms}.
\newblock {\em Inventiones Mathematicae}, 171(2):257--279, 2008.

\bibitem{EMS}
A.~Eskin, S.~Mozes, and N.~Shah.
\newblock Non-divergence of translates of certain algebraic measures.
\newblock {\em Geom. Funct. Anal.}, 7(1):48--80, 1997.

\bibitem{gorodnik-frames}
A.~Gorodnik.
\newblock Uniform distribution of orbits of lattices on spaces of frames.
\newblock {\em Duke Math. J.}, 122:549--589, 2004.

\bibitem{gorodnik}
A.~Gorodnik.
\newblock {Uniform distribution of orbits of lattices on spaces of frames}.
\newblock {\em Duke Math. J}, 122(3):549--489, 2004.

\bibitem{gorodnik-oh}
A.~Gorodnik and H.~Oh.
\newblock Equidistribution of adelic periods.

\bibitem{goroweiss}
A.~Gorodnik and B.~Weiss.
\newblock Distribution of lattice orbits on homogeneous varieties.
\newblock {\em \`a paraitre dans {G}eometric and functional analysis}, 2004.

\bibitem{mathese}
A.~Guilloux.
\newblock {\em \'Equidistribution dans les espaces homog\`enes}.
\newblock PhD thesis, Universit\'e Paris 11 Orsay, 2007.

\bibitem{kleinbock-tomanov}
Dmitry Kleinbock and George Tomanov.
\newblock Flows on {$S$}-arithmetic homogeneous spaces and applications to
  metric {D}iophantine approximation.
\newblock {\em Comment. Math. Helv.}, 82(3):519--581, 2007.

\bibitem{ledrappier}
F.~Ledrappier.
\newblock Distribution des orbites des r\'eseaux sur le plan r\'eel.
\newblock {\em C. R. Acad. Sci}, 329:61--64, 1999.

\bibitem{ledrappier-pollicott}
F.~Ledrappier and M.~Pollicott.
\newblock Distribution results for lattices in ${SL}(2,\dr {Q}_p)$.
\newblock {\em Bull. Braz. Math. Soc.}, 36(2):143--176, 2005.

\bibitem{margulis}
G.A. Margulis.
\newblock {\em Discrete subgroups of semisimple {L}ie groups}.
\newblock Springer, 1991.

\bibitem{maucourant}
F.~Maucourant.
\newblock Homogeneous asymptotic limits of {H}aar measures of semisimple lie
  groups and their lattices.
\newblock {\em Duke Math. J.}, \`A para\^itre.

\bibitem{nogueira}
A.~Nogueira.
\newblock {Orbit distribution on $\dr R^2$ under the natural action of
  $\d{SL}(2, \dr Z)$}.
\newblock {\em Indagationes Mathematicae}, 13(1):103--124, 2002.

\bibitem{platonov-rapinchuk}
V.~Platonov and A.~Rapinchuk.
\newblock {\em Algebraic Groups and Number Theory}.
\newblock Academic Press, Boston MA, London, Sydney, 1994.

\bibitem{shah}
N.A. Shah.
\newblock Limit distributions of expanding translates of certain orbits on
  homogeneous spaces.
\newblock {\em Proc. Indian Acad. Sci. (Math. Sci.)}, 106 no 2:105--125, 1996.

\bibitem{Tits}
J.~Tits.
\newblock Reductive groups over local fields.
\newblock {\em Proceedings of Symposia in Pure Mathematics}, 33:20--70, 1979.

\bibitem{Tomanov1}
G.~Tomanov.
\newblock Orbits on homogeneous spaces of arithmetic origin and approximations.
\newblock {\em Adv. Stud. Pure Math.}, 26:265--297, 2000.

\end{thebibliography}
\end{document}